\documentclass[a4paper,11pt]{amsart}

\usepackage[margin=2.9cm, ]{geometry}
\usepackage{amsmath}
\usepackage{amsfonts}
\usepackage{amscd, amssymb}
\usepackage{amsmath,amscd}
\usepackage[latin2]{inputenc}
\usepackage[T1]{fontenc}
\usepackage{ae}
\usepackage{aecompl}
\usepackage[colorlinks=false, citebordercolor={0 1 0}, linkbordercolor={1 0 0}]{hyperref}
\usepackage{graphics}
\usepackage{graphicx}
\usepackage{bm}
\usepackage{enumitem}
\usepackage{mathtools}
\setlist[enumerate,1]{label={(\alph*)}}
\setlist[enumerate,2]{label={(\roman*)}}

\newtheorem{thm}{Theorem}[section]
\newtheorem{prop}[thm]{Proposition}

\newtheorem{lem}[thm]{Lemma}

\newtheorem{lemma}[thm]{Lemma}

\newtheorem{algorithm}[thm]{Algorithm}

\newtheorem{ques}[thm]{Question}
\newtheorem{conj}[thm]{Conjecture}

\newtheorem*{notation*}{Notation}

\theoremstyle{definition}
\newtheorem{definition}[thm]{Definition}

\theoremstyle{remark}
\newtheorem{rmk}[thm]{Remark}

\newcommand{\ignore}[1]{}
\newcommand{\R}{\mathbb R}

\newcommand{\N}{\mathbb N}
\newcommand{\Prob}{\mathbb P}

\newcommand{\E}{{\mathbb{E}}}

\newcommand{\termdef}[1]{{\textbf{#1}}}
\newcommand{\oone}{{o \left(1\right)}}
\newcommand{\omegaone}{{\omega \left(1\right)}}
\newcommand{\hkn}{{H_k \left(n\right)}}
\newcommand{\hnp}{{\mathcal{H}_k \left( n ; p \right)}}
\newcommand{\steinParams}{{\left( n , k , k - 1 \right)}}
\newcommand{\F}{{\mathbb{F}}}
\newcommand{\FO}{{\mathcal{F}}}
\newcommand{\VO}{{\mathcal{V}}}
\newcommand{\EO}{{\mathcal{E}}}
\newcommand{\given}{{|}}

\DeclareMathOperator{\im}{Im}

\DeclareMathOperator{\rk}{rk}

\begin{document}
\title{$\steinParams$-Steiner Systems in Random Hypergraphs}
\author{Michael Simkin}
\address{Institute of Mathematics and Federmann Center for the Study of Rationality,
	The Hebrew University of Jerusalem, Jerusalem 91904, Israel}
\email{menahem.simkin@mail.huji.ac.il}

\dedicatory{Dedicated to Noam Yonat Rosen. ``Only one is my dove, my perfect one'' (Song of Songs 6:9 \cite{jewish1985tanakh}).}

\begin{abstract}
Let $H$ be a random $k$-uniform $n$-vertex hypergraph where every $k$-tuple belongs to $H$ independently with probability $p$. We show that for some $\varepsilon_k > 0$, if $p \geq n^{-\varepsilon_k}$, then asymptotically almost surely $H$ contains an $\steinParams$-Steiner System. Our main tool is Keevash's method of Randomized Algebraic Constructions.
\end{abstract}

\maketitle

\pagestyle{plain}

\section{Introduction}\label{sec:intro}

One of Erd\H{o}s and R\'enyi's first discoveries in the theory of random $G \left( n ; p \right)$ graphs \cite{erdHos1968random} is that the threshold for the appearance of a perfect matching is $p=\frac{\log n}{n}$. Consider a random $k$-uniform $n$-vertex hypergraph in which every $k$-set is included with probability $p$. Schmidt and Shamir \cite{schmidt1983threshold} asked what the threshold is for the existence of a perfect matching in this model. Their problem attracted a lot of attention and was eventually settled in a seminal paper of Johansson, Kahn, and Vu \cite{johansson2008factors} who showed that the threshold probability is $p = \frac{\log n}{n^{k - 1}}$ (the same as the threshold for the disappearance of isolated vertices).

However, a perfect matching in a $k$-uniform hypergraph is not the only natural analogue of a perfect matching in a graph. It is, in our mind, just as natural to ask about the emergence of an $\steinParams$-Steiner system, namely, a $k$-uniform hypergraph on $\left[ n \right]$ s.t.\ every $\left( k - 1 \right)$-size subset is contained in precisely one edge. Note that for $k=2$ this is just a (graphical) perfect matching.

Unlike perfect matchings, the mere existence of $\steinParams$-Steiner systems for infinitely many $n \in \N$ and $k \geq 5$ was a very long-standing open question until the recent breakthrough work of Peter Keevash \cite{Ke14} (see Theorem \ref{thm:Keevash}, below). He proved that for all but finitely many $n$ satisfying necessary divisibility conditions (see Section \ref{ssec:notation} below for the precise conditions) an $\steinParams$ Steiner system exists. An alternative proof was subsequently given by Glock, K\"uhn, Lo, and Osthus \cite{glock2016existence}. This makes it meaningful and interesting to consider the threshold for the appearance of Steiner systems in random hypergraphs.

Denote by $\hkn$ the set of $k$-uniform hypergraphs on vertex set $\left[ n \right]$. For $p \in \left[ 0,1 \right]$ let $\hnp$ be the distribution on $\hkn$ where each edge is included in the hypergraph with probability $p$, independently. If $\emptyset \neq \mathcal{P} \subsetneq \hkn$ is a monotone increasing property then there exists some unique $p_c = p_c \left( \mathcal{P} \right) \in \left[ 0,1 \right]$ s.t.\ $\Prob_{H \sim \mathcal{H} \left( n ; p_c\right)} \left[ H \in \mathcal{P} \right] = \frac{1}{2}$. If $I$ is an infinite subset of $\N$ and $\mathcal{P} = \left\{ \mathcal{P}_n \right\}_{n \in I}$ is a sequence of nontrivial, monotone increasing properties where for all $n \in I, \mathcal{P}_n \subseteq \hkn$, $p_c \left( \mathcal{P}_n \right)$ is the \termdef{threshold function} for $\mathcal{P}$ (for other, equivalent, definitions of the threshold function see \cite{bollobas1987threshold}). Let $N_k$ be the (infinite) set of integers $n$ for which an $\steinParams$-Steiner system exists. We ask:

\begin{ques}
	What is the threshold function for the property that $H \sim \hnp$ contains an $\steinParams$-Steiner system, where $n$ ranges over $N_k$?
\end{ques}

An obvious necessary condition for a hypergraph to contain a Steiner system is that every $\left(k-1\right)$-set must be contained in at least one edge. This provides a lower bound of $\Omega \left( \frac{\log n }{n} \right)$ on the threshold function. We don't know of any way to improve this and this might well be the correct threshold (see Section \ref{sec:future} for related conjectures). Our main result is an upper bound on the threshold function, which to our knowledge is new:

\begin{thm}\label{thm:main}
	For every $k \geq 2$ there exists some $\varepsilon > 0$ s.t.\ if $n \in N_k$ and $p \geq n^{-\varepsilon} $ then the probability that $H \sim \hnp$ contains an $\steinParams$-Steiner system is $1 - n^{-\omegaone}$.
\end{thm}

\begin{rmk}
	It's natural to ask how $\varepsilon$ depends on $k$ in our proof, and how close $n^{-\varepsilon}$ is to the threshold function. We believe that even the best $\varepsilon$ obtainable by our methods is far from the truth; therefore we've chosen a simpler analysis over optimal results. As to the value of $\varepsilon$ we do obtain, it will be easier to discuss after proving Theorem \ref{thm:main}. We do so in Section \ref{sec:future}.
\end{rmk}

\begin{rmk}
	The proof of Theorem \ref{thm:main} is by a probabilistic algorithm that, given $H \sim \hnp, p = n^{-\varepsilon}$ for sufficiently small $\varepsilon$, finds an $\steinParams$-Steiner system contained in $H$ with high probability. The proof of Keevash's main theorems in \cite{Ke14} and \cite{Ke15} is similar. In fact, two steps of our algorithm, constructing the ``Template'' and ``Nibble'' (in Sections \ref{ssec:template} and \ref{ssec:nibble}, respectively), are ``lifted'' straight from Keevash's proof of the main theorem in \cite{Ke14}. Another step (described in Section \ref{ssec:keevash_application}) is a direct application of Keevash's Theorem (\cite{Ke15} Theorem 6.2, quoted below as Theorem \ref{thm:Keevash}) guaranteeing the existence of a $K_k$-decomposition of pseudo-random $\left(k-1\right)$-uniform hypergraphs (to be defined below).

	In light of the reliance on Keevash's techniques one might wonder if Theorem \ref{thm:main} isn't a more direct consequence of Keevash's work. In fact, it seems that a weaker version of Theorem \ref{thm:main}, with $n^{-\varepsilon}$ replaced by a function tending to $0$ at an unknown rate, can be derived straightforwardly from \cite{Ke14} Theorem 6.1. Alternatively, it seems likely that Theorem \ref{thm:main} can be proved by following the proof of \cite{Ke15} Theorem 6.2 almost to the letter, with the additional constraint that only $k$-sets from some $H \sim \mathcal{H}_k \left( n ; n^{-\varepsilon} \right)$ (with $\varepsilon$ sufficiently small) may be used in the proof (though we haven't verified this).
	
	Despite the points raised in the previous paragraph we feel it is better to give a self-contained proof, modulo known results. Since we use Keevash's theorem as a ``black box'', the reader need not verify any specific details of his proof, which is quite complex. Furthermore, while our proof indeed uses techniques from Keevash's proof it is significantly less complicated, in that it relies only on elementary linear algebra and measure concentration while avoiding the need to induct on $k$ and the topic of integral designs, which play prominent roles in \cite{Ke14} and \cite{Ke15}.
\end{rmk}

The remainder of this paper is organized as follows: The rest of section \ref{sec:intro} sets up notation and terminology, and quotes several measure concentration results. In section \ref{sec:proof_k4} we prove Theorem \ref{thm:main} for $k \leq 4$. In section \ref{sec:proof_k5} we explain how to modify the proof for the slightly more complicated case where $k \geq 5$. In section \ref{sec:generalization} we discuss generalizations to Latin squares, one-factorizations, and high-dimensional permutations. In section \ref{sec:future} we present some conjectures and related problems.

\subsection{Notation and Terminology}\label{ssec:notation}

Let $\hkn$ be the set of $k$-uniform hypergraphs on vertex set $\left[ n \right]$. Let $\hnp$ be the distribution on $\hkn$ where each edge is included with probability $p$, independently. We identify ${H \in \hkn}$ with its edge set, viewing $H$ as a subset of $\binom{\left[n\right]}{k}$. We write $d \left( H \right) = \frac{\left| H \right|}{\binom{n}{k}}$. We write $K_k^n$ for the complete $k$-uniform hypergraph on $\left[ n \right]$, i.e.\ $K_k^n = \binom{\left[ n \right]}{k}$. Borrowing from the language of polytopes, if $n$ and $k$ are clear (for example, if we have fixed some $H \in H_k \left(n\right)$), we refer to an element of $\binom{\left[n\right]}{k-1}$ as a \termdef{facet}.

Although the edges of a hypergraph are sets, we sometimes need to fix an ordering of the elements of an edge $x = \left\{ x_1 , x_2 , \ldots , x_k \right\} \in K_k^n$. In such circumstances we write $x = x_1 x_2 \ldots x_k$. This is to be understood as an arbitrary ordering that, once chosen, remains fixed.

For $H \in H_{k-1} \left(n\right)$, $K_k \left( H \right) \coloneqq \left\{ e \in \binom{\left[n\right]}{k} : \binom{e}{k-1} \subseteq H \right\} \in H_k \left(n\right)$. For example, if $G$ is a graph, $K_3 \left(G\right)$ is the set of induced triangles. For $H \in \hkn$,
$$
{K_{k-1} \left( H \right) = \cup_{e \in H} \binom{e}{k-1} \in H_{k-1} \left( n \right)}
$$
is the set of facets contained in the edges of $H$, and is referred to as the set of $H$'s facets.

A \termdef{$K_k$-decomposition} of $G \in H_{k - 1} \left( n \right)$ is a graph $H \subseteq K_k \left( G \right)$ s.t.\ every facet in $G$ is contained in precisely one edge in $H$. There are necessary divisibility conditions $G$ must satisfy for it to have a $K_k$-decomposition. Specifically, for $S \subseteq \left[ n \right]$, we write $G \left( S \right) = \left\{ f \subseteq \left[ n \right] \setminus S : S \cup f \in G \right\}$. If $S \in \binom{\left[ n \right]}{i}$, $0 \leq i \leq k-1$, then, on the one hand, there are $\left| G \left( S \right) \right|$ edges containing $S$, and on the other hand each $k$-set containing $S$ contains $k-i$ edges containing $S$. Therefore, if a $K_k$-decomposition of $G$ exists, we must have $k-i | \left| G \left( S \right) \right|$. We say that $G$ is \termdef{$k$-divisible} if for all $0 \leq i \leq k-1$ and all $S \in \binom{\left[ n \right]}{i}$, $\binom{k-i}{k-1-i}$ divides $\left| G \left( S \right) \right|$.

A $K_k$-decomposition of $K_{k-1}^n$ is precisely an \termdef{$\steinParams$-Steiner system}. In this case, the conditions for $k$-divisibility translate to the arithmetic constraints
$$
{\forall 0 \leq i \leq k-1, k-i | \binom{n-i}{k-1-i}}
$$
If $n$ satisfies these conditions we say it is \termdef{$k$-divisible}. A consequence of Keevash's main Theorem in \cite{Ke14} is that $N_k$, the set of numbers $n$ for which $\steinParams$-Steiner systems exist, is all but finitely many of the $k$-divisible numbers. If $n$ and $k$ are clear from context, we refer to an $\steinParams$-Steiner system as simply a \termdef{Steiner system} or a \termdef{design}.

$S,T \in K_k^n$ are \termdef{facet-disjoint} if $K_{k-1} \left( S \right) \cap K_{k-1} \left( T \right) = \emptyset$. We say that $S$ is \termdef{facet-disjoint} if for all $x,y \in S$, $x \neq y \implies \left| x \cap y \right| < k - 1$. We sometimes call a facet-disjoint set a \termdef{partial design}. A facet $f$ is \termdef{covered} by a partial design $S$ if there exists some $e \in S$ containing $f$.

If $G \in \hkn$, $\F$ is a field, and $\pi : \left[ n \right] \rightarrow \F$, we denote the set 
$$
\left\{ \left( \pi x_1 , \pi x_2 , \ldots , \pi x_k \right)^T : \left\{ x_1 x_2 \ldots x_k \right\} \in G \right\} \subseteq \F^k
$$
by $\pi \left( G \right)$.

A sequence of events $\left\{ E_n \right\}_{n \in \N}$ occurs \termdef{with very high probability} (\termdef{w.v.h.p.}) if ${\Prob \left[ E_n \right] = 1 - n^{-\omegaone}}$. If we have polynomially many (sequences of) events, each occurring w.v.h.p., then their intersection also occurs w.v.h.p. For the most part we will abuse terminology and refer to a single event as occurring w.v.h.p., where the sequential structure and the dependence on $n$ may be inferred.

For $a,b \in \R$ we write $a \pm b$ to indicate a quantity in the interval $\left[ a - \left|b\right| , a + \left|b\right| \right]$.

\subsection{Measure Concentration}

Throughout the proof we'll calculate the expectations of various random variables and then argue that w.v.h.p.\ they are, in fact, close to their mean. To this end we collect several concentration inequalities.

\begin{lem}[\cite{kwan2016almost} Theorem 2.11]\label{lem:boolean_concentration}
	Let $f : \left\{ 0, 1 \right\}^N \rightarrow \R$ be a Boolean function and $C \in \R$ s.t.\ for all $x,y \in \left\{ 0,1 \right\}^N$ that differ in exactly one coordinate, $\left| f \left( x \right) - f \left( y \right) \right| \leq C$ (i.e.\ $f$ is $C$-Lipschitz). Let $X = \left( X_1 , X_2 , \ldots , X_N \right)$ be a sequence of independent Bernoulli variables with mean $q \in \left[ 0,1 \right]$. Then for all $t > 0$:
	$$
	\Prob \left[ \left| f \left( X \right) - \E f\left( X \right) \right| > t \right] \leq \exp \left( - \frac{t^2}{4 C^2 N q + 2 C t} \right)
	$$
\end{lem}

\begin{lem}[\cite{Ke14} Theorem 2.10]\label{lem:azuma_concentration}
	Let $f : \left\{ 0,1 \right\} ^N \rightarrow \R$ be a Boolean function and $b_1 , b_2 , \ldots , b_N >0$ where if $x,y \in \left\{ 0,1 \right\} ^N$ differ only on the $i$th coordinate, ${\left| f \left( x \right) - f \left( y \right) \right| \leq b_i}$. Let $X_1 , X_2 , \ldots , X_N$ be independent Bernoulli random variables. Then for every $t > 0$:
	$$
	\Prob \left[ \left| f \left( X \right) - \E f \left( X \right) \right| > t \right] \leq 2 \exp \left( - \frac{t^2}{2 \sum_{i=1}^N b_i^2} \right)
	$$
\end{lem}

For two sets $X,Y$ we write $Sym \left( X \right)$ for the group of bijections from $X$ to itself and $\mathcal{I} \left( X,Y \right)$ for the set of injections from $X$ to $Y$. The next Lemma is an immediate consequence of \cite{Ke14} Lemma 2.13.

\begin{lem}\label{lem:random_injection}
	Let $X,Y$ be finite sets and let $C > 0$. Suppose $f:\mathcal{I} \left( X,Y \right) \rightarrow \R$ has the property that for every $\pi \in \mathcal{I}$ and every transposition $\tau \in Sym \left( Y \right)$, ${\left| f \left( \pi \right) - f \left( \tau \circ \pi \right) \right| \leq C}$. Let $\pi$ be a uniformly random element of $\mathcal{I} \left( X,Y \right)$. Then for every $t > 0$:
	$$
	\Prob \left[ \left| f \left( \pi \right) - \E f \left( \pi \right) \right| > t \right] \leq 2 \exp \left( - \frac{t^2}{2 \left| Y \right| C^2} \right)
	$$
\end{lem}

\begin{lem}[\cite{Ke14} Lemma 2.7]\label{lem:stochastic_bound}
	Let $ \mathcal{F}_0 \subseteq \mathcal{F}_1 \subseteq \ldots \subseteq \mathcal{F}_n $ be a filtration of a finite probability space. Let $Y_1 , Y_2 , \ldots , Y_n$ be a sequence of random variables and $C \in \R$ s.t.\ for every $i$, $Y_i$ is $\mathcal{F}_i$-measurable, $\left| Y_i \right| \leq C$, and $\E \left[ \left| Y_i \right| \given \mathcal{F}_{_i-1} \right] \leq \mu_i$. Suppose $Y = \sum_{i=1}^n Y_i$ and $\mu > \sum_{i=1}^n \mu_i$. Then for all $c > 0$:
	$$
	\Prob \left[ \left| Y \right| > \left( 1 + c \right) \mu \right] \leq 2 \exp \left( - \frac{\mu c^2}{2 \left( 1 + 2c \right) C} \right)
	$$ 
\end{lem}

\section{Proof of Theorem \ref{thm:main} for $k \leq 4$}\label{sec:proof_k4}

The $k=2$ case follows from the classical result of Erd\H{o}s and R\'enyi so we may assume $k \in \left\{ 3,4 \right\}$.

Throughout the section we'll assume w.l.o.g.\ that $p = n^{-\varepsilon}$. Wherever necessary we assume $n \in N_k$ is arbitrarily large and $\varepsilon$ is arbitrarily small (but independent of $n$). Asymptotic terms refer to fixed $k$ and $n$ tending to infinity.

Let $H \sim \hnp$. Our goal is to find a $K_k$-decomposition of $K_{k-1}^n$ contained in $H$. We do so by by means of a probabilistic algorithm proceeding roughly as follows, where each step can be completed w.v.h.p.:

\begin{itemize}
	\item \textbf{Template} (Section \ref{ssec:template}): Set aside a partial design $T \subseteq H$ that has desirable switching properties. Later in the proof we'll construct a design that may use a small number of edges not in $H$. These edges will be replaced by ``algebraic absorbers'' induced by the template.
	
	\item \textbf{Nibble} (Section \ref{ssec:nibble}): Using a random greedy algorithm construct a partial design $N \subseteq H$ that is facet-disjoint from $T$ and covers almost all facets.
	
	\item \textbf{Applying Keevash's Theorem} (Section \ref{ssec:keevash_application}): Let $L \in H_{k-1}\left(n\right)$ be the collection of facets uncovered by $T \cup N$. Apply Keevash's existence theorem from \cite{Ke15} to obtain a $K_k$-decomposition ${S \subseteq K_k \left( L \right)}$ of $L$. Typically, $S$ uses edges that aren't in $H$.
	
	\item \textbf{Absorbing $S$} (Section \ref{ssec:absorbing}): Using algebraic absorbers replace the edges in $S$ (as well as some edges from $T$) with edges from $H$, yielding a $K_k$-decomposition of $K_{k-1}^n$ that is contained in $H$.
	
\end{itemize}

\subsection{Template}\label{ssec:template}

Let $2n \leq 2^m \leq 4n$. Let $\gamma = \frac{n}{2^m - 1} = \Theta \left( 1 \right)$. Let $\F$ be the field with $2^m$ elements (so that $\left| \F \right| = \Theta \left( n \right)$). We remark that since $\F$ has characteristic $2$, if $u,v$ are elements in a vector space over $\F$ then $v = -v$ and $u+v = 0 \iff u = v$. Furthermore, if $V$ is a vector space over $\F$ of dimension $d = O \left( 1 \right)$ then $\left| V \right| = \Theta \left( n^d \right)$. If $T : V \rightarrow U$ is a linear map, then $\left| \ker T \right| \left| \im T \right| = \Theta \left( n^d \right)$. Let $\F^* = \F \setminus \left\{ 0 \right\}$.

Let $\pi : \left[n\right] \hookrightarrow \F^*$ be a uniformly random injection. For the remainder of the proof we identify $x \in \left[ n \right]$ with $\pi \left(x\right)$, and view $\left[ n \right]$ as a subset of $\F^*$. We define the \termdef{template} $T$:
$$
T = \left\{ x_1 x_2 \ldots x_k \in H : x_1 + x_2 + \ldots + x_k = 0 \right\}
$$

The edges in $T$ are called \termdef{algebraic}. A facet $f \in K_{k-1}^n$ is \termdef{algebraic} if for some (unique) $x \in \left[n\right]$, $f \cup \left\{ x \right\}$ is algebraic.

The remainder of Section \ref{ssec:template} establishes properties of $T$ and $H$ that hold w.v.h.p. Section \ref{sssec:cp} introduces cross-polytopes - small hypergraphs that form the building blocks of the absorbers. Section \ref{sssec:absorbers} introduces the absorbers themselves - small subhypergraphs of $H$ whose vertices satisfy certain linear constraints. Section \ref{sssec:linear_operators} defines families of linear operators related to the absorbers and establishes their salient properties. These operators provide a convenient framework within which we may study the interactions between absorbers. Finally, Section \ref{sssec:template_properties} establishes the properties of $T$ and $H$ necessary to continue with the proof.

\subsubsection{Cross-Polytopes}\label{sssec:cp}

Let $x_1 , x_2 ,\ldots ,x_k , a_1 , a_2 ,\ldots , a_k \in \left[ n \right]$ be distinct vertices. The \termdef{cross-polytope spanned by} $x_1 , x_2 ,\ldots ,x_k , a_1 , a_2 ,\ldots , a_k$ and denoted $C_{x_1 , x_2 ,\ldots ,x_k , a_1 , a_2 ,\ldots , a_k}$ is the $k$-uniform hypergraph on vertex set $\left\{ x_1 , x_2 ,\ldots ,x_k , a_1 , a_2 ,\ldots , a_k \right\}$ with the edges:
$$
e_I \coloneqq \left\{ x_i : i \in \left[ k \right] \setminus I \right\} \cup \left\{ a_i : i \in I \right\}, I \subseteq \left[ k \right]
$$
In other words, the edges consist of the sets containing exactly one of $x_i$ and $a_i$ for each $i$. We refer to the edges $\left\{ e_I \right\}_{I \subseteq \left[ k \right]}$ as either \termdef{even} or \termdef{odd}, depending on the parity of $\left| I \right|$. Observe that both
$$
C_{x_1 , x_2 ,\ldots ,x_k , a_1 , a_2 ,\ldots , a_k}^{even} =
\left\{ e_I : I \subseteq \left[ k \right], \left|I\right|\ even \right\}
$$
\and
$$
C_{x_1 , x_2 ,\ldots ,x_k , a_1 , a_2 ,\ldots , a_k}^{odd} =
\left\{ e_I : I \subseteq \left[ k \right], \left|I\right|\ odd \right\}
$$
are $K_k$-decompositions of $K_{k-1} \left( C_{x_1 , x_2 ,\ldots ,x_k , a_1 , a_2 ,\ldots , a_k} \right)$.

If  ${x = x_1 x_2 \ldots x_k \in K_k^n}$ is non-algebraic with the property that every facet in $x$ is algebraic, then there exist $a = a_1 , a_2 , \ldots ,a_k \in K_k^n$ s.t.\ for every $i \in \left[ k \right]$, ${x_1 x_2 \ldots x_{i-1} a_i x_{i+1} \ldots x_k \in T}$. Observe that $a_i = \sum_{j=1}^n x_j + x_i$ and that the $a_i$s are distinct. Indeed, if $i \neq j$:

\begin{equation*}
\begin{gathered}
a_i + a_j = 2 \sum_{\ell = 1}^n x_\ell + x_i + x_j = x_i + x_j \neq 0
\end{gathered}
\end{equation*}

We call $C_{x,a}$ the \termdef{associated cross-polytope} of $x$, and denote it by $C_x$. We refer to the vertices $a_1 , a_2 , \ldots , a_k$ as the \termdef{algebraic vertices} of $C_x$.  Observe that $C_x^{odd} \subseteq T$; we call this the \termdef{algebraic decomposition} of $C_x$. We call $C_x^{even}$ the \termdef{non-algebraic decomposition} of $C_x$.

Denote by $\mathcal{C} \subseteq K_k^n$ the set of $x \in K_k^n$ s.t.\ $x$ is non-algebraic but every facet in $x$ is (so that the associated cross-polytope $C_x$ exists), and $C_x^{even} \setminus \left\{ x \right\} \subseteq H$. Note that every edge in the algebraic decomposition $C_x^{odd}$ is necessarily in $T \subseteq H$; hence $C_x^{odd} \subseteq H$.

\subsubsection{Absorbers}\label{sssec:absorbers}

In Section \ref{sssec:template_properties} we'll show that w.v.h.p.\ (over the choice of $H$ and $\pi$) almost every $s \in K_k^n$ can be embedded in many absorbers. These have the following structure:

Let $x = x_1 x_2 \ldots x_k \in K_k^n$. An edge $a = a_1 , a_2 , \ldots , a_k \in K_k^n$ spans an \termdef{absorber} for $x$ if:
\begin{itemize}
	\item $x \cap a = \emptyset$.
	
	\item $C_{ x , a}^{odd} \subseteq H$.
	
	\item $C_{ x , a}^{even} \setminus \left\{ x \right\} \subseteq \mathcal{C}$. In particular, if $\emptyset \neq I \subseteq \left[ k \right]$ has even cardinality, then $e_I \in C_{x,a}$ has an associated cross polytope $C_{e_I}$. Furthermore, all edges in $C_{e_I}$ are contained in $H$, with the possible exception of $e_I$ itself.	
	
	\item $x_1 , x_2 , \ldots , x_k , a_1 , a_2 , \ldots , a_k$, and the vertices of all the associated cross-polytopes $C_{e}$, where $x \neq e \in C_{x,a}^{even}$, are all distinct. This ensures that all the associated cross-polytopes are facet disjoint.
\end{itemize}
	
The absorber spanned by $x,a$ is the $k$-uniform hypergraph
$$
C_{x,a}^{odd} \cup \bigcup_{e \in C_{x,a}^{even} \setminus \left\{ x \right\}} \left( C_{e_I} \setminus \left\{ e_I \right\} \right)
$$
The image one should keep in mind is that an absorber for $x$ is a cross-polytope containing $x$ in which on every even edge (with the exception of $x$) lies an associated cross-polytope.

Denote by $\mathcal{A}_x$ the set of absorbers for $x$.
	
If $A \in \mathcal{A}_x$ is spanned by $a \in K_k^n$, the \termdef{algebraic decomposition} of $A$ is the hypergraph:
$$
A^{alg} = \bigcup_{e \in C_{x,a}^{even} \setminus \left\{ x \right\} } C_e^{odd} \subseteq T
$$
The \termdef{non-algebraic decomposition} of $A$ is the hypergraph:
$$
A^{non-alg} = C_{x,a}^{odd} \cup \left( \bigcup_{e \in C_{x,a}^{even} \setminus \left\{ x \right\}} \left(C_e^{even} \setminus \left\{ e \right\}\right) \right)
\subseteq H
$$
Observe that $A^{alg}$ is a $K_k$-decomposition of $K_{k-1} \left( A \right) \setminus K_{k-1} \left( x \right)$, and that $A^{non-alg}$ is a $K_k$-decomposition of $K_{k-1} \left( A \right)$.
	
Finally, set:
$$
M = M \left( k \right) \coloneqq \left| A \right| = \left| A^{alg} \right| + \left| A^{non-alg} \right| = 2^{2k-1} - 2^k + 1
$$

\subsubsection{Vertex, Facet, and Edge Operators}\label{sssec:linear_operators}
\newcommand{\xa}{{\begin{pmatrix} x \\ a \end{pmatrix}}}

An important property of absorbers is that their vertices, facets, and edges can be seen as linear functions of the spanning vertices. This allows us to estimate the number of absorbers for a given edge, and to control the ways in which absorbers for different edges intersect.

In what follows, we sometimes abuse terminology and refer to a vector as what properly should be a set. For example, we may refer to an element of $\F^{k-1}$ as a facet. This should be understood as the set of coordinates of the vector. In the other direction, we sometimes refer to a set by what properly is a vector. For example, we may say that an edge $x \in K_k^n$ is an element of $\F^k$. This should be understood to mean that at the first mention, we fix an ordering of the elements of $x$ so that it is indeed a vector. All subsequent statements should hold regardless of the particular ordering.

Let $T_1 ,T_2 , \ldots , T_k : \F^k \rightarrow \F$ and $P_1, P_2, \ldots, P_k: \F^k \rightarrow \F^{k-1}$ be the canonical projections, i.e., $T_i \left(x_1 , x_2 , \ldots , x_k \right)^T = x_i$ and $P_i \left( x_1 , x_2 , \ldots , x_k \right)^T = \left( x_1 ,\ldots ,x_{i-1} , x_{i+1} , \ldots x_k \right)^T$. Observe that if $x \in K_k^n$ the vertices of $x$ are precisely $T_1 x , T_2 x , \ldots , T_k x$ and the facets contained in $x$ are precisely $P_1 x , P_2 x , \ldots , P_k x $.

Define the linear operator $C:\F^k \rightarrow \F^{2k}$ by:
\begin{equation*}
	C =
	\begin{bmatrix}
		I_k \\
		J_k + I_k
	\end{bmatrix}
\end{equation*}
Where $I_k \in M_k \left( \F \right)$ is the identity matrix and $J_k \in M_k \left( \F \right)$ is the all $1$s matrix. If $x \in \mathcal{C}$, $Cx$ is the vertex set of the associated cross-polytope $C_x$. For every $I \subseteq \left[ k \right]$, define the linear operator $E_I: \F^{2k} \rightarrow \F^k$ by:
\begin{equation*}
	E_I =
	\begin{bmatrix}
		X_{\left[ k \right] \setminus I} & X_{I}
	\end{bmatrix}
\end{equation*}
Where $X_I \in M_k \left( \F \right)$ is $1$ on the positions $ \left( i,i \right), i \in I$ and $0$ elsewhere. Notice that if $x,a \in \F^k$ span a cross-polytope, its edges are precisely $\left\{ E_I \begin{pmatrix}
x \\ a \end{pmatrix} : I \subseteq \left[ k \right] \right\}$.

Combining the observations above, for $x \in \mathcal{C}$, $\left\{ P_i E_I C x : i \in \left[ k \right], I \subseteq \left[ k \right] \right\}$ is the collection of facets in $C_x$. In fact, since every facet is contained in both an edge $e_I$ with $\left| I \right|$ even and an edge $e_I$ with $ \left| I \right|$ odd, we may restrict ourselves to $I \subseteq \left[ k \right]$ of a specific parity. Therefore, if we define:
\begin{gather*}
\begin{align*}
	\FO_1 &= \left\{ P_i E_{\emptyset} : i \in \left[ k \right] \right\}\\
	\FO_2 &= \left\{ P_i E_J C E_I : i \in \left[ k \right] , J \subseteq \left[ k \right] , \left| J \right| = 0 \bmod 2 , \emptyset \neq I \subseteq \left[ k \right] , \left| I \right| = 0 \bmod 2 \right\}\\
	\FO &= \FO_1 \cup \FO_2
\end{align*}
\end{gather*}
then the collection of facets in the absorber for $x$ spanned by $a \in \F^k$ is:
$$
\left\{ F \begin{pmatrix} x \\ a \end{pmatrix} : F \in \FO \right\}
$$

We call $\FO$ the family of \termdef{facet operators}. The facets given by $\FO_1$ are those contained in $x$, with the remaining facets given by $\FO_2$.

We define the \termdef{vertex operators} $\VO$ similarly. Let $\VO_1 = \left\{ T_i E_\emptyset , T_i E_{\left[ k \right]} : i \in \left[ k \right] \right\}$, $\VO_2 = \left\{ T_i E_{ \left[ k \right] } C E_I : i \in \left[ k \right] , \emptyset \neq I \subseteq \left[ k \right] , \left| I \right| = 0 \bmod 2 \right\}$, and $\VO = \VO_1 \cup \VO_2$. If $a \in \F^k$ spans an absorber for $x$, its vertex set is:
$$
\left\{ V \begin{pmatrix} x \\ a \end{pmatrix} : V \in \VO \right\}
$$

Finally, we define the \termdef{edge operators}. Let:
\begin{gather*}
\begin{align*}
\EO_1 &= \left\{ E_I : I \subseteq \left[ k \right] , \left| I \right| = 1 \bmod 2 \right\}\\
\EO_2 &= \left\{ E_J C E_I : \emptyset \neq J \subseteq \left[ k \right], \emptyset \neq I \subseteq \left[ k \right], \left| I \right| = 0 \bmod 2 \right\}\\
\EO &= \mathcal{E}_1 \cup \mathcal{E}_2
\end{align*}
\end{gather*}
If $x,a \in \F^k$ span an absorber, its edge set is:
$$
\left\{ E \begin{pmatrix} x \\ a \end{pmatrix} : E \in \EO \right\}
$$

\begin{notation*}
	For a linear map $T$ defined on $\F^{2k}$ we write $T_1 , T_2$ for the unique linear maps defined on $\F^k$ s.t.\ for all $x = \left( x_1 , x_2 , \ldots , x_{2k} \right)^T \in \F^{2k}$, $Tx = T_1 \left( x_1 , \ldots , x_k \right)^T + T_2 \left( x_{k+1} , \ldots , x_{2k} \right)^T$.
\end{notation*}

Before discussing these operators' properties we pause for motivation. The point of the linear operators is to allow us to answer questions such as: for $x,x' \in K_k^n$, how many pairs of absorbers $A \in \mathcal{A}_x , A' \in \mathcal{A}_{x'}$ are there that intersect each other on a facet? Well, this quantity is certainly bounded above by the number of pairs $a,a' \in \F^k$ for which there exist $F,F' \in \FO$ s.t.:
\begin{equation}\label{eq:facet_intersect}
F \begin{pmatrix} x \\ a \end{pmatrix} + F' \begin{pmatrix} x' \\ a' \end{pmatrix} = 0
\end{equation}
Changing our perspective, we may ask: for a given pair $F,F' \in \FO$, how many pairs $a,a' \in \F^k$ exists s.t.\ Equation \ref{eq:facet_intersect} holds? Well, $\left\{ F \begin{pmatrix} x \\ a \end{pmatrix} + F' \begin{pmatrix} x' \\ a \end{pmatrix} \right\}_{a , a' \in \F^k}$ is an affine subspace of $\F^{k-1}$, and we are asking about the size of the inverse image of $\left\{ 0 \right\}$ under given affine maps. Linear algebra is the perfect framework for this type of calculation.

We now establish useful properties of $\VO, \EO$, and $\FO$.

\begin{prop}\label{pr:vo_properties}
	Let $x \in K_k^n$. There are $O \left( n^{k-1} \right)$ vectors $a \in \F^k$ s.t.\ either $0 \in \left\{ V \xa : V \in \VO \right\}$ or the values $\left\{ V \xa \right\}_{V \in \VO}$ are not distinct. As a consequence there are $\Omega \left( n^k \right)$ edges $a \in K_k^n$ s.t.\ the values $\left\{ V \xa \right\}_{V \in \VO}$ are all distinct and non-zero.
\end{prop}

\begin{rmk}
	The proof of Proposition \ref{pr:vo_properties} uses the fact that $x \in K_k^n$ doesn't contain any subset of $k-2$ vertices that sum to $0$. This holds only because $k \leq 4$. This is the only place in the proof where this is used. If $x$ doesn't contain any $\left(k-2\right)$-set that sums to $0$ then the conclusion holds even when $k \geq 5$.
\end{rmk}

\begin{proof}

	We first observe that if $T : \F^{2k} \rightarrow \F$ is a linear functional and $x \in \F^k$ is fixed, then $\left\{ a \in \F^k : T \xa = 0 \right\}$ is an affine subspace of $\F^k$. If it isn't the entire space then its dimension is at most $k-1$ and its size is $O \left( n^{k-1} \right)$. To show that the latter is the case it's enough to exhibit some $a \in \F^k$ s.t.\ $T \xa \neq 0$. Alternatively we may show:
	\begin{equation}\label{eq:null_image}
	T_1 x \neq 0 \lor T_2 \neq 0
	\end{equation}
	
	Let $x \in K_k^n$. In order to prove the proposition it's enough to show that Equation \ref{eq:null_image} holds for each of the linear maps $V \in \VO$ and $V+V'$ where $V,V' \in \VO$ are distinct.
	
	First, if $V \in \VO_1$, then by definition of $\VO_1$ we have either $V_1 = T_i$ or $V_2 = T_i$ for some $i \in \left[ k \right]$, both of which imply Equation \ref{eq:null_image}.
	
	If $V \in \VO_2$, then for some $i \in \left[ k \right] , \emptyset \neq I \subseteq \left[k \right] , \left| I \right| = 0 \bmod 2$, $V = T_i E_{\left[ k \right]} C E_I$. Let $j \in I \setminus \left\{ i \right\}$. Then $V_2 e_j = T_i E_{\left[ k \right]} C X_I e_j = 1 \neq 0$ which implies Equation \ref{eq:null_image}.
	
	It remains to show that Equation \ref{eq:null_image} holds for $V+V'$, where $V,V' \in \VO$ are distinct. We consider several cases:
	
	\begin{itemize}
		\item $V , V' \in \VO_1$. In this case $\left( V + V' \right) \xa$ is the difference between two elements of $x_1 , \ldots , x_k , a_1 , \ldots , a_k$, which are distinct for all but $O \left( n^{k-1} \right)$ choices of $a$.
		
		\item $V = T_i E_\emptyset \in \VO_1$ for some $i$ and $V' = T_j E_{\left[ k \right]} C E_I \in \VO_2$, for some $j$ and even sized non-empty $I \subseteq \left[ k \right]$. Then $V_2 + V'_2 = T_j E_{\left[ k \right]} C X_I$. Let $\ell \in I \setminus \left\{ j \right\}$. Then $\left( V_2 + V'_2 \right)e_\ell = 1$, implying that $V_2 + V'_2 \neq 0$, so Equation \ref{eq:null_image} holds.
		
		\item $V = T_i E_{\left[ k \right]} \in \VO_1$ for some $i$ and $V' = T_j E_{\left[ k \right]} C E_I \in \VO_2$, for some $j$ and even sized non-empty $I \subseteq \left[ k \right]$. If $I \setminus \left\{ i,j \right\}$ is non-empty then by an argument similar to the previous case $V_2 + V'_2 \neq 0$. Otherwise $I = \left\{ i,j \right\}$, so $\left( V_1 + V'_1 \right) x = T_j E_{\left[ k \right]} C X_{\left[ k \right] \setminus \left\{ i,j \right\}} x = \sum_{\ell \in \left[ k \right] \setminus \left\{ i,j \right\}} x_\ell$. This is the sum of $k-2$ non-zero distinct elements of $\F^*$, and so is non-zero. Therefore Equation \ref{eq:null_image} holds.
		
		\item $V , V' \in \VO_2$, and $V = T_i E_{\left[ k \right]} C E_I , V' = T_j E_{\left[ k \right]} C E_I$ for appropriate $i,j$ and $I$. In this case $V \xa + V' \xa = \left( T_j + T_i \right) E_I$, which is the sum of two of $x_1 , \ldots , x_k , a_1 , \ldots , a_k$. For all but $O \left( n^{k-1} \right)$ choices for $a$ these are distinct, implying that Equation \ref{eq:null_image} holds.
		
		\item $V , V' \in \VO_2$, $V = T_i E_{\left[ k \right]} C E_I$, and $V' = T_j E_{\left[ k \right]} C E_J$ for $I \neq J$. Note that ${\left| I \Delta J \right| \geq 2}$. If there is some $\ell \in I \setminus \left( J \cup \left\{ i \right\} \right)$ then $\left( V_2 + V'_2 \right) e_\ell = 1$ and we're done. The same holds if $J \setminus \left( I \cup \left\{ j \right\} \right) \neq \emptyset$. Otherwise we have $I \Delta J = \left\{ i,j \right\}$, and in particular $i \neq j$. But then $\left( V_1 + V'_1 \right) x = \sum_{\ell \in \left[ k \right] \setminus I} x_\ell + \sum_{\ell \in \left[ k \right] \setminus J} x_\ell = x_i + x_j \neq 0$, so Equation \ref{eq:null_image} holds.
	\end{itemize}
\end{proof}

The next two propositions, whose proofs are omitted, follow from similar examinations of the definitions of $\FO$ and $\EO$.

\begin{prop}\label{pr:fo_properties}
	For every $F \in \FO$ the following hold:
	\begin{enumerate}
		\item\label{itm:map_sum} If $F \in \FO_2$ there exist distinct $i , j \in \left[ k \right]$ s.t.\ $e_i , e_j \in \ker F_1$.

\iffalse		
		\item $\rk F = k - 1$.
		
		\item\label{itm:ker_dim} If $F \in \FO_2$ then $\dim \ker F_1 \geq 2$.
\fi
		
		\item\label{itm:im_dim} If $F \in \FO_2$ then $\rk F_2 \geq 1$.
		
		\item\label{itm:fo_inverse_space} For every $f \in \F^{k-1}$, $\left\{  a \in \F^k : F_1^{-1} \left( f + F_2 a \right) \neq \emptyset \right\} \subseteq \F^k$ is an affine subspace of dimension $\rk F_1 + 1$.
	\end{enumerate}

\end{prop}

\begin{prop}\label{pr:eo_properties}
	For every $E \in \EO$, $\dim \ker E_2 \leq k-1$.
\end{prop}

%Proof is boring and technical, so commented out.
\iffalse
\begin{proof}
	It's enough to show that there exists some $i \in \left[ k \right]$ s.t.\ $E_2 e_i \neq 0$.
		
	If $E \in \EO_1$ then there exists some non-empty $I \subseteq \left[ k \right]$ s.t.\ $E = E_I$. Then $E_2 = X_I$. Let $i \in I$. Then $E_2 e_i = X_{I} e_i = e_i \neq 0$.
		
	Otherwise $E \in \EO_2$, so there exist $\emptyset \neq J \subseteq \left[ k \right], \emptyset \neq I \subseteq \left[ k \right], \left| I \right| = 0 \bmod 2$ s.t.\ $E = E_J C E_{I}$. Hence $E_2 = E_J C X_{I}$. Let $i,j \in I$ be distinct. If $i \notin J$ then $E_2 e_i = E_J C X_I e_i = e_i \neq 0$. Otherwise the $ith$ coordinate of $E_2 e_j$ satisfies $\left(E_2 e_j\right)_i = \left(E_J C X_I e_j\right)_i = 1 \neq 0$.
\end{proof}
\fi

\subsubsection{Properties of the Template}\label{sssec:template_properties}

The next Proposition establishes that every edge in $K_k^n$ can be embedded in many absorbers. Recall that $\mathcal{A}_x$ is the set of absorbers for $x \in K_k^n$ and that $M$ is the number of edges in an absorber.

\begin{prop}\label{pr:absorbers}
	W.v.h.p.\ (over the choice of $H$ and $\pi$) for every $x \in K_k^n$, $\left| \mathcal{A}_x \right| = \Omega \left( n^{ k - M \varepsilon } \right)$.
\end{prop}

\begin{rmk}
	The proof relies only on the conclusion of Proposition \ref{pr:vo_properties} holding for $x$. Therefore, if $x$ contains no $\left(k-2\right)$-set whose sum is $0$, w.v.h.p.\ there are $\Omega \left( n^{k - M \varepsilon } \right)$ absorbers for $x$ even if $k \geq 5$.
\end{rmk}

\begin{proof}
	It's enough to show that the bound holds w.v.h.p.\ for an arbitrary $x \in K_k^n$.
	
	As $H$ and $\pi$ are independent, we may apply the following two-step analysis: Condition on the values $\pi$ takes on $x$. Let $\mathcal{A}_x'$ be the set of $a \in \left( \F^* \right)^k$ satisfying the conditions in Proposition \ref{pr:vo_properties} (so $\left| \mathcal{A}'_x \right| = \Omega \left( n^k \right)$). For each $a \in \mathcal{A}'_x$, $\left| \left\{ V \xa : V \in \VO \right\} \right| = \left| \VO \right| = O \left( 1 \right)$. Therefore the probability (over $\pi \left( \left[ n \right] \setminus x \right)$) of all vertices in the absorber spanned by $x , a$ being in $\pi \left( \left[ n \right] \right)$ is at least $\left( 1 - \oone \right) \gamma^{\left| \VO \right|} = \Omega \left( 1 \right)$. Let $\mathcal{A}''_x$ be the set of $a \in \mathcal{A}'_x$ s.t.\ $\left\{ V \begin{pmatrix} x \\ a \end{pmatrix} : V \in \VO \right\} \subseteq \pi \left( \left[ n \right] \right)$. Then:
	$$
	\E_{\pi \left( \left[ n \right] \setminus x \right)} \left[ \left| \mathcal{A}''_x \right| \right] = \left| \mathcal{A}'_x \right| \left( 1 - \oone \right) \gamma^{ k \left( 2^{k-1} + 1 \right) }
	= \Omega \left( n^k \right)
	$$
	Composing a transposition of $\F^* \setminus \pi \left( x \right)$ with $\pi$ affects $\left| \mathcal{A}_x' \right|$ by at most $O \left( n^{k-1} \right)$. Thus, applying Lemma \ref{lem:random_injection}, w.v.h.p.\ $\left| \mathcal{A}_x'' \right| = \Omega \left( n^{k} \right)$.
	
	For the second step, to obtain $\mathcal{A}_x$, we choose $H \sim \hnp$. Note that for $a \in \mathcal{A}_x''$, since the vertices $\left\{ V \xa \right\}_{V \in \VO}$ are distinct, so are the facets $\left\{ F \xa \right\}_{F \in \FO}$. Therefore $a \in \mathcal{A}''_x$ is also in $\mathcal{A}_x$ iff all of the $M \left( k \right)$ edges of the would-be absorber spanned by $x,a$ are in $H$. Thus:
	$$
	\E_H \left[ \left| \mathcal{A}_x \right| \right] = p^{M \left( k \right)} \left| \mathcal{A}_x'' \right|
	= \Omega \left( n^{k - M\left( k \right) \varepsilon} \right)
	$$
	
	$\left| \mathcal{A}_x \right|$ is a Boolean function of the independent Bernoulli random variables $\left\{ \chi_e \right\}_{e \in K_k^n}$ where $\chi_e = 1$ iff $e \in H$. We'd like to apply Lemma \ref{lem:azuma_concentration} to conclude that w.v.h.p.\ $\left| \mathcal{A}_x \right|$ is close to its expectation. For $e \in K_k^n$, set:
	$$
	b_e = \sum_{E \in \EO : e \in E_1 x + \im E_2} \left| \ker E_2 \right|
	$$
	Adding or removing $e$ from $H$ alters $\left| \mathcal{A}_x \right|$ by at most $b_e$. We have:
	\begin{gather*}
	\begin{align*}
		B &\coloneqq \sum_{e \in K_k^n} b_e^2
		= \sum_{e \in K_k^n} \left( \sum_{E \in \EO : e \in E_1 x + \im E_2} \left| \ker E_2 \right| \right)^2
		= O \left( \sum_{e \in K_k^n} \sum_{E \in \EO : e \in E_1 x + \im E_2} \left| \ker E_2 \right|^2 \right)\\
		&= O \left( \sum_{E \in \EO} \sum_{e \in E_1 x + \im E_2} \left| \ker E_2 \right|^2 \right)
		= O \left( \sum_{E \in \EO} \left| \im E_2 \right| \left| \ker E_2 \right|^2 \right)
		= O \left( n^k \sum_{E \in \EO} \left| \ker E_2 \right| \right)
	\end{align*}
	\end{gather*}
	By Proposition \ref{pr:eo_properties} $\dim \ker E_2 \leq k-1 \implies \left| \ker E_2 \right| = O \left( n^{k-1} \right)$. Thus:
	$$
	B = O \left( n^{2k-1} \right)
	$$
	Applying Lemma \ref{lem:azuma_concentration} with (for example) $t = \frac{1}{2} \E \left| \mathcal{A}_x \right| = \Omega \left( n^{k - M \varepsilon } \right)$ we conclude that w.v.h.p.\ $\left| \mathcal{A}_x \right| = \Omega \left( n^{k - M \varepsilon} \right)$.
\end{proof}

We proceed conditioning on the conclusions of Proposition \ref{pr:absorbers} holding.

\subsection{Nibble}\label{ssec:nibble}

The goal of this stage is to extend $T$ to a partial design $T \cup N \subseteq H$ covering almost all facets, such that the hypergraph $L \coloneqq K_{k-1}^n \setminus K_{k-1} \left( T \cup N \right)$ of uncovered facets is pseudo-random (in a sense to be made precise) and not too sparse (specifically, it will have density at least $n^{-a}$, where $a = a \left( k \right) > 0$ is a small constant to be specified later). There are several avenues we can follow. We sketch two approaches, both involving probabilistic proofs.

The first is to apply the random greedy packing algorithm to $H$: Set $N = \emptyset$. Let $L \in H_{k-1} \left( n \right)$ be the set of facets uncovered by $T \cup N$. As long as $d \left( L \right) \geq n^{-a}$ and $H$ contains edges that are facet disjoint from $T \cup N$, choose one uniformly at random and add it to $N$, then update $L$ accordingly.

The second approach is to use the R\"{o}dl Nibble: As before, set $N = \emptyset$ and let $L$ be the set of facets uncovered by $T \cup N$. Let $G \subseteq H$ be the set of edges that are facet disjoint from $T \cup N$. Let $\eta = \eta \left( k \right) > 0$ be a sufficiently small constant. As long as $d \left( L \right) \geq n^{-a}$ and $G \neq \emptyset$, choose a random subset $S \subseteq G$ where each edge is selected with independent probability $\frac{\eta}{n d \left( G \right)}$. Of the edges in $S$, keep only those that are facet disjoint from all others, and add them to $N$. Update $L$ and $G$ accordingly.

We now define the pseudo-randomness conditions we'll need.

\begin{definition}
	Let $S \subseteq \F^{k-1}$ and let $C > 0$. $S$ is \termdef{$C$-affine-bounded} if for every affine space $A \subseteq \F^{k-1}$ of dimension at least $1$, ${\left| A \cap S \right| \leq C \left| A \right| \frac{\left| S \right|}{\left| \F \right|^{k-1} }}$. If $G \in H_{k-1}^n$ and $\tau : \left[ n \right] \rightarrow \F$, we say that $\left( G , \tau \right)$ is $C$-affine-bounded if $\tau \left( G \right)$ is $C$-affine-bounded. If $\tau$ and $C$ are implicit, we simply say that $G$ is affine-bounded.
\end{definition}

\begin{definition}
	Let $G \in H_{k-1}^n$, $c > 0$, and $h \in \N$. $G$ is \termdef{$\left(c,h\right)$-typical} if for any collection of $\ell \leq h$ distinct $S_1 ,\ldots ,S_\ell \in \binom{\left[ n \right]}{k-2}$, ${\left| \cap_{i \in \left[ \ell \right]} G \left( S_i \right) \right| = \left( 1 \pm c \right) d \left( G \right)^\ell n}$.
	
\end{definition}

Roughly speaking, a hypergraph is affine-bounded if its intersection with every non-trivial affine space (under an appropriate map of the vertices) is bounded above by what we'd expect in a random hypergraph of the same density. Similarly, typicality means that all the degrees, codegrees, etc., are what we'd expect in a random hypergraph.

The next lemma establishes that affine-boundedness carries over, in a certain sense, to $K_k$-decompositions.

\begin{lemma}\label{lem:bounded_inverse}
	There exists some $D = D \left( k \right) > 0$ s.t.\ the following holds: Let $G \in H_{k-1} \left( n \right)$, $\tau : \left[ n \right] \hookrightarrow \F^*$ be injective, and $C > 0$. Let $S \subseteq K_k \left( G \right)$ be a $K_k$-decomposition of $G$. If $ \left(G , \tau \right)$ is $C$-affine-bounded then for any $f \in \F^{k-1}$ and $F \in \FO_2$, $\left| \tau \left( S \right) \cap F_1^{-1} \left( f \right) \right| \leq D C \frac{\left| S \right|}{n^k} \left| F_1^{-1} \left( f \right) \right|$.
\end{lemma}

\begin{proof}
	
	If $f \notin \im F_1$ then $F_1^{-1} \left( f \right) = \emptyset$ and the conclusion holds trivially. Otherwise $\left| F_1^{-1} \left( f \right) \right| = \left| \ker F_1 \right|$. Furthermore, $F_1^{-1} \left( f \right) = x_0 + \ker F_1$ for some $x_0 \in \F^k$.
	
	By Proposition \ref{pr:fo_properties} item \ref{itm:map_sum}, there exist distinct $i , j \in \left[ k \right]$ s.t.\ $e_i , e_j \in \ker F_1$. Observe that $e_i \in \ker P_i$. Therefore $\dim P_i \left( \ker F_1 \right) \leq \dim \ker F_1 - 1$. Additionally, $P_i e_j \neq 0$, so $\dim P_i \left( \ker F_1 \right) \geq 1$. Hence:
	$$
	1 \leq \dim P_i \left( F_1^{-1} \left( f \right) \right) = \dim P_i \left( \ker F_1 \right) \leq \dim \ker F_1 - 1
	$$
	Because $S$ is a $K_k$-decomposition of $G$, $P_i$ induces an injection from $\tau \left( S \right)$ into $\tau \left( G \right)$. Therefore:
	$$
	\left| \tau \left( S \right) \cap F_1^{-1} \left( f \right) \right|
	= \left| P_i \left( \tau \left( S \right) \cap F_1^{-1} \left( f \right) \right) \right|
	\leq \left| \tau \left( G \right) \cap P_i \left( F_1^{-1} \left( f \right) \right) \right|
	$$
	Applying affine-boundedness and the observations above:
	\begin{gather*}
	\begin{align*}
	\left| \tau \left( G \right) \cap P_i \left( F_1^{-1} \left( f \right) \right) \right|
	= O \left( C \left| P_i \left( F_1^{-1} \left( f \right) \right) \right| \frac{\left| G \right|}{n^{k-1}} \right)
	= O \left( C n^{\dim \ker F_1 - 1} \frac{ \left| G \right| }{n^{k-1}} \right)
	\end{align*}
	\end{gather*}
	Finally, since $\left| S \right| = \Theta \left( \left| G \right| \right)$, we have:
	$$
	\left| \tau \left( G \right) \cap P_i \left( F_1^{-1} \left( f \right) \right) \right| = O \left( C \frac{\left| S \right|}{n^k} \left| F_1^{-1} \left( f \right) \right| \right)
	$$
	We may take $D$ to be the implicit constant in the final term.
\end{proof}

\begin{prop}\label{pr:leave}
	Fix $h , \ell \in \N$, and $a, c_0 > 0$. There exist $C , \delta > 0$ s.t.\ w.v.h.p.\ (over the choice of $H$ and $\pi$) there exists some $N \subseteq H$ facet-disjoint from $T$ s.t.\ $L \coloneqq K_{k-1}^n \setminus K_{k-1} \left( N \cup T \right)$ satisfies:
	\begin{enumerate}
		\item\label{itm:affine_bounded_leave} $L$ is $C$-affine-bounded.
		
		\item $L$ is $\left( c,h \right)$-typical, with $c < c_0 d \left( L \right)^\ell$.
		
		\item $n^{-\delta} \geq d \left( L \right) \geq n^{- a}$.
		
		\item $L$ is $k$-divisible.
	\end{enumerate}
\end{prop}

Both the random greedy packing algorithm and the Nibble yield $N$ and $L$ with the desired properties w.v.h.p.\ (over the internal randomness of the algorithms as well as $H$ and $\pi$). Both algorithms have been extensively analyzed (for example, see \cite{rodl1985packing} or \cite{alon2016probabilistic}, Chapter 4.7, for the Nibble, and \cite{Wo99} Section 7.2 for the greedy packing algorithm) and it is straightforward to adapt these analyses to prove Proposition \ref{pr:leave}. For completeness' sake we analyze the Nibble in Appendix \ref{app:nibble}.

We proceed conditioning on the existence of $N,L$, and $\delta$ as in Proposition \ref{pr:leave}.

\subsection{Applying Keevash's Theorem}\label{ssec:keevash_application}

\newcommand{\leave}{{K_{k - 1} \left( S \setminus \left\{ s_i \right\} \right)}}

We recall Keevash's existence theorem (\cite{Ke15}, Theorem 6.2) as it applies to $\steinParams$-Steiner systems:

\begin{thm}\label{thm:Keevash}
	For every $k \geq 2$ there exists $c_0 , a \in \left( 0,1 \right)$ and $h, \ell , n_0 \in \N$ s.t.\ if $n \geq n_0$ and $G \in H_{k-1} \left( n \right)$ is $k$-divisible, $\left( c,h \right)$-typical, $d \left( G \right) > n^{-a}$, and $c < c_0 d \left( G \right)^\ell$ then $K_k \left( G \right)$ has a $K_k$-decomposition.
\end{thm}

By Proposition \ref{pr:leave} $L$ satisfies the conditions in Theorem \ref{thm:Keevash}, and hence it has a $K_k$-decomposition ${S \subseteq K_k \left( L \right)}$.

\subsection{Absorbing $S$}\label{ssec:absorbing}

Note that $T \cup N \cup S$ is a $K_k$-decomposition of $K_{k-1}^n$. Typically, however, $S \nsubseteq H$. To remedy this, we'll replace the edges in $S$ with appropriate absorbers, as follows:

\begin{algorithm}\label{alg:RGA}\leavevmode
	\begin{itemize}
		\item Order the the edges in $S$ arbitrarily, $s_1 , s_2 , \ldots ,s_{\left| S \right|}$.
		
		\item For every $1 \leq i \leq \left| S \right|$, choose $A_i$ uniformly at random from the absorbers in $\mathcal{A}_{s_i}$ that are facet-disjoint from $A_1 , A_2 , \ldots , A_{i-1}$ as well as $S \setminus \left\{ s_i \right\}$. If there are no such absorbers, abort.
	\end{itemize}
\end{algorithm}

If the algorithm doesn't abort then
$$N \cup \left( T \setminus \left( \bigcup_{i=1}^{\left| S \right|} A_i^{alg} \right) \right) \cup \left( \bigcup_{i=1}^{\left| S \right|} A_i^{non-alg} \right) \subseteq H$$
is a $K_k$-decomposition of $K_{k-1}^n$.

It remains to prove that w.v.h.p.\ Algorithm \ref{alg:RGA} doesn't abort. To get some feel for why this is true, consider the last step of the algorithm, in which we choose an absorber for $s_{\left|S\right|}$. At this point $\left|S\right| - 1 = O \left( n^{k-1 - \delta} \right)$ absorbers have been selected, each covering $O \left( 1 \right)$ facets. We adopt the heuristic that the set of covered facets is a random subset of $K_{k-1}^n$ of size $O \left( n^{k-1-\delta} \right)$. This will be justified by affine-boundedness of $L$ together with Lemma \ref{lem:bounded_inverse}. Now, consider how many absorbers for $s_{\left|S\right|}$ a facet $f$ excludes if it is covered by a different absorber. If, for example, $f \in F_1 s_{\left|S\right|} + \im F_2$ for some $F \in \FO_2$, then it may exclude as many as $\left| \ker F_2 \right|$ absorbers. If we sum over all elements of $F_1 s_{\left|S\right|} + \im F_2$, the expected number of absorbers excluded in this way is $O \left( \frac{n^{k-1-\delta}}{n^{k-1}} \left| \im F_2 \right| \left| \ker F_2 \right| \right) = O \left( n^{k - \delta} \right)$. The last term is independent of $F$, so we conclude that the expected number of excluded absorbers is $O \left( n^{k - \delta} \right)$. Provided $\varepsilon$ is small enough, this is much less than $\left| \mathcal{A}_{s_{\left|S\right|}} \right|$.

We'll need an estimate for the number $O_f$ of absorbers for the elements of $S$ that cover a given facet $f \in K_{k-1}^n \setminus L$.

\begin{lemma}\label{lem:absorber_cover}
	Let $f \in K_{k-1}^n \setminus L$. Let $O_f = \sum_{s \in S} \left| \left\{ A \in \mathcal{A}_s : f \in K_{k-1} \left( A \right) \right\} \right|$. Then $O_f = O \left( n^{k-\delta} \right)$.
\end{lemma}

\begin{proof}
	By definition:
	$$
	O_f \leq \sum_{s \in S} \sum_{F \in \FO} \left| \left\{ a \in \F^k : F_1 s + F_2 a = f \right\} \right|
	$$
	If $F \in \FO_1$ then, by the definition of $\FO_1$, for all $s \in S$, $F_1 s + F_2 a = F_1 s \in \pi \left( L \right)$. But by assumption $f \notin L$, hence $F_1 s + F_2 a \neq f$ and this contributes nothing to the sum. Changing the order of summation we're left with:
	$$
	O_f \leq \sum_{F \in \FO_2} \sum_{s \in S} \left| \left\{ a \in \F^k : F_1 s + F_2 a = f \right\} \right|
	$$
	Rearranging:
	$$
	O_f
	\leq \sum_{F \in \FO_2} \sum_{a \in \F^k} \left| \left\{ s \in S : F_1 s + F_2 a = f \right\} \right|
	$$
	Now, for any $F \in \FO_2$ and $a \in \F^k$:
	$$
	\left\{ s \in S : F_1 s + F_2 a = f \right\} \subseteq \pi \left( S \right) \cap {F_1}^{-1} \left( f + F_2 a \right)
	$$
	Since $\pi \left( L \right)$ is affine-bounded and $S$ is a $K_k$-decomposition of $L$ we have, by Lemma \ref{lem:bounded_inverse}:
	$$
	\left| \pi \left( S \right) \cap {F_1}^{-1} \left( f + F_2 a \right) \right|
	= O \left( \frac{\left| L \right|}{n^k} \left| {F_1}^{-1} \left( f + F_2 a \right) \right| \right)
	$$
	Thus:
	$$
	O_f = \sum_{F \in \FO_2} \sum_{a \in \F^k : {F_1}^{-1} \left( f + F_2 a \right) \neq \emptyset } O \left( \frac{\left| L \right|}{n^k} \left| {F_1}^{-1} \left( f + F_2 a \right) \right| \right)
	$$
	By Proposition \ref{pr:fo_properties} item \ref{itm:fo_inverse_space} $\left\{ a \in \F^k : {F_1}^{-1} \left( f + F_2 a \right) \neq \emptyset \right\}$ is an affine space of dimension $\rk F_1 + 1$. For such $a$, $\left| {F_1}^{-1} \left( f + F_2 a \right) \right| = \left| \ker F_1 \right| = O \left( n^{k - \rk F_1} \right)$. Therefore:
	$$
	O_f = \sum_{F \in \FO_2} O \left( n^{\rk F_1 + 1} \frac{\left| L \right|}{n^k} n^{k - \rk F_1} \right) = O \left( n \left| L \right|\right) = O \left( n^{k - \delta} \right)
	$$
\end{proof}

\begin{prop}\label{pr:RGA}
	W.v.h.p.\ (over the internal randomness of the algorithm) Algorithm \ref{alg:RGA} doesn't abort.
\end{prop}

\begin{proof}
	\newcommand{\pil}{{\pi \left( L \right)}}
	
	Write $\mathcal{A}_i \coloneqq \mathcal{A}_{s_i}$. Recall that by Proposition \ref{pr:absorbers} $\left| \mathcal{A}_i \right| = \Omega \left( n^{k - M \varepsilon} \right)$. Given $A_1 , A_2 , \ldots , A_{ i-1 }$ we'll call the absorbers in $\mathcal{A}_i$ that are facet disjoint from $A_1 , A_2 , \ldots , A_{i-1}$ and $S \setminus \left\{ s_i \right\}$ \termdef{permissible} and the remaining absorbers \termdef{excluded}. We'll show that w.v.h.p.\ at every step of the algorithm there are at least $\frac{1}{2} \left| \mathcal{A}_i \right| = \Omega \left( n^{k - M \varepsilon} \right)$ permissible absorbers. In particular, w.v.h.p.\ the algorithm doesn't abort.
	
	We first show that for every $1 \leq i \leq \left| S \right|$ all but $o \left( \left| \mathcal{A}_i \right| \right)$ absorbers in $\mathcal{A}_i$ are facet disjoint from $S \setminus \left\{ s_i \right\}$. Let $B_i$ be the number of absorbers for $s_i$ that intersect $S \setminus \left\{ s_i \right\}$ on a facet, i.e.:
	$$
	B_i = \left| \left\{ A \in \mathcal{A}_i : K_{k-1} \left( A \right) \cap K_{k-1} \left( S \setminus \left\{s_i \right\} \right) \neq \emptyset \right\} \right|
	$$
	We have:
	\begin{equation*}
	\begin{aligned}
	B_i
	& \leq \sum_{F \in \FO} \left| \left\{ a \in \F^k : F \begin{pmatrix} s_i \\ a \end{pmatrix} \in K_{k-1} \left( S \setminus \left\{ s_i \right\} \right) \right\} \right| \\
	& \leq \sum_{F \in \FO} \left| \ker F_2 \right| \left| \left( F_1 s_i + \im F_2 \right) \cap \pi \left( K_{k - 1} \left( S \setminus \left\{ s_i \right\} \right) \right) \right|
	\end{aligned}
	\end{equation*}
	Now, for every $F \in \FO_1$, $F_1 s_i + \im F_2 = F_1 s_i \in K_{k-1} \left( s_i \right)$ and, since $S$ is a collection of facet disjoint edges, this is disjoint from $\leave$. So we're left with:
	\begin{gather*}
		\begin{align*}
			B_i &\leq \sum_{F \in \FO_2} \left| \ker F_2 \right| \left| \left( F_1 s_i + \im F_2 \right) \cap \pi \left( K_{k - 1} \left( S \setminus \left\{ s_i \right\} \right) \right) \right|\\
			&\leq \sum_{F \in \FO_2} \left| \ker F_2 \right| \left| \left( F_1 s_i + \im F_2 \right) \cap \pi \left( L \right) \right|
		\end{align*}
	\end{gather*}
	By Proposition \ref{pr:fo_properties} item \ref{itm:im_dim} $F_1 s_i + \im F_2$ is an affine space of dimension at least $1$. Since $\pi \left( L \right)$ is affine-bounded, we have:
	$$
	B_i
	= O \left( \sum_{F \in \FO_2} \left| \ker F_2 \right| \left| \im F_2 \right| \frac{\left| L \right|}{n^{k - 1}} \right)
	= O \left( n^{k - \delta} \right)
	= o \left( \left| \mathcal{A}_i \right| \right)
	$$
	The last equality holds provided $\varepsilon < \frac{\delta}{M}$.
	
	Next we demonstrate that w.v.h.p.\ the choices of $A_1 , A_2 , \ldots , A_{i-1}$ leave many permissible absorbers for $s_i$. We'll define the stopping time $\tau$ as the smallest $i$ s.t.\ there are fewer than $\frac{\left| \mathcal{A}_i \right|}{2}$ permissible absorbers in $\mathcal{A}_i$ and $\tau = \infty$ if there is no such $i$. It's enough to show that for each $i$:
	\begin{equation}\label{eq:many_absorbers}
	\Pr \left[ \tau = i \given \tau \geq i \right] = n^{-\omegaone}
	\end{equation}
	
	Fix $ 1 \leq i \leq \left| S \right|$ and assume, inductively, that Equation \ref{eq:many_absorbers} holds for $j < i$. We'd like to bound (w.v.h.p.) the number of absorbers in $\mathcal{A}_i$ excluded by the choices of $A_1 , A_2 , \ldots , A_{i-1}$. For each $F \in \FO$, let $E_F$ be the number of elements $a \in \F^k$ s.t.\ $f \coloneqq F_1 \left( s_i \right) + F_2 \left( a \right) \notin \pil$ and $f$ has been covered by one of $A_1 ,\ldots , A_{i-1}$. The number of excluded absorbers is bounded above by $B_i + \sum_{F \in \FO} E_F$. We'll show that w.v.h.p.\ $E_F = o \left( \left| \mathcal{A}_{s_i} \right|\right)$ for every $F$.
	
	First, if $F \in \FO_1$, then for every $a \in \F^k$, $F \begin{pmatrix} s_i \\ a \end{pmatrix} = F_1 s_i \in \pil$, so by definition $E_F = 0$. Otherwise $F \in \FO_2$. We'll bound $E_F$ by an application of Lemma \ref{lem:stochastic_bound}. The filtration is the successive choices of $A_1 , \ldots , A_{i-1}$, and we bear in mind that we've conditioned on $\tau \geq i$. For $f \in \F^{k-1} \setminus \pi \left( L \right)$, and $j<i$ let $X_f^j$ be the indicator of the event that $f$ is covered by $A_j$. Then:
	$$
	E_F \leq \sum_{j<i} \left| \ker F_2 \right| \sum_{f \in F_1 s_i + \im F_2} X_f^j
	$$
	For $j < i$ set:
	$$
	E_F^j = \left| \ker F_2 \right| \sum_{f \in F_1 s_i + \im F_2} X_f^j
	$$
	We next bound the maximal conditional expectation. Let $j < i$. In the following, the maximum is taken over all choices $A_1 , A_2 ,  \ldots , A_{j-1}$ s.t.\ $\Prob \left[ A_1 , \ldots , A_{j-1} \
	\given \tau \geq i \right] > 0$. By the inductive hypothesis this is a non-empty set:
	\begin{gather*}
	\begin{align*}
		\mu_j & \coloneqq \max_{A_1 , A_2 , \ldots , A_{j-1}} \E \left[ E_F^j \given A_1 , A_2 , \ldots A_{j-1} , \tau \geq i \right] \\
		& \leq \max_{A_1 , A_2 , \ldots , A_{j-1}} \frac{1}{\Prob \left[ \tau \geq i \right]} \E \left[ E_F^j \given A_1 , \ldots , A_{j-1} , \tau > j \right]
	\end{align*}
	\end{gather*}
	By the inductive hypothesis $\Prob \left[ \tau \geq i \right] = 1 - \oone$. Furthermore, $\tau > j$ implies that there are at least $\frac{1}{2} \left| \mathcal{A}_j \right| = \Omega \left( n^{k - M \varepsilon} \right)$ permissible absorbers for $s_j$. Therefore:
	$$
	\mu_j \leq O \left(\frac{\left| \ker F_2 \right|}{n^{k - M \varepsilon}}\right) \sum_{f \in F_1 s_i + \im F_2} \left| \left\{ A \in \mathcal{A}_j : f \in K_{k-1} \left( A \right) \right\} \right|
	$$
	Hence:
	\begin{gather*}
		\begin{align*}
			\mu \coloneqq \sum_{j<i} \mu_j
			&\leq O \left(\frac{\left| \ker F_2 \right|}{n^{k - M \varepsilon}}\right) \sum_{f \in F_1 s_i + \im F_2} \sum_{j < i} \left| \left\{ A \in \mathcal{A}_j : f \in K_{k-1} \left( A \right) \right\} \right|\\
			&\leq O \left(\frac{\left| \ker F_2 \right|}{n^{k - M \varepsilon}}\right) \sum_{f \in F_1 s_i + \im F_2} O_f
		\end{align*}
	\end{gather*}
	By Lemma \ref{lem:absorber_cover} $O_f = O \left( n^{k-\delta} \right)$. Therefore:
	$$
	\mu = O \left( \frac{\left| \ker F_2 \right|}{n^{k- M \varepsilon}} \right) \left| \im F_2 \right| O \left( n^{k- \delta } \right)
	= O \left( n^{k + M \varepsilon - \delta} \right)
	$$
	Since every absorber covers only $O \left( 1 \right)$ facets, regardless of the choice of $A_j$, $E_F^j = O \left( \left| \ker F_2 \right| \right)$. By Proposition \ref{pr:fo_properties} item \ref{itm:im_dim} $\dim \ker F_2 = k - \dim \im F_2 \leq k - 1$, so $E_F^j = O \left( n^{k-1} \right)$. Applying Lemma \ref{lem:stochastic_bound} we conclude that w.v.h.p.:
	$$
	E_F = O \left( n^{k + M \varepsilon - \delta} \right) = o \left( n^{k - M \varepsilon} \right)
	$$
	Provided $\varepsilon < \frac{\delta}{2 M}$, the last equality implies $E_F = o \left( \left| \mathcal{A}_i \right| \right)$, and Equation \ref{eq:many_absorbers} holds.
	
\end{proof}

\section{Modifying the Proof for $k \geq 5$}\label{sec:proof_k5}

The crucial part of the proof in the previous section was the fact, established in Proposition \ref{pr:absorbers}, that every edge can be embedded in many absorbers. In turn, this relied on Proposition \ref{pr:vo_properties} which showed that for every $x \in K_k^n$ there are $\Omega \left( n^k \right)$ vectors $a \in \F^k$ for which $\left\{ V \xa \right\}_{V \in \VO}$ are distinct and non-zero. This used the fact that if $k \leq 4$ there are no sets of $k-2$ distinct elements of $\F^*$ whose sum is $0$. This is not true for $k \geq 5$. In order to ensure that every edge has many absorbers we modify the proof by taking a richer template, as follows:

Let $\pi_1 , \pi_2 , \ldots , \pi_{k+1} : \left[ n \right] \hookrightarrow \F^*$ be independent uniformly random injections. We define a sequence of facet-disjoint templates and take $T$ to be their union:
\begin{align*}
\forall j \in \left[ k+1 \right], & T_j = \left\{ x \in H : \sum_{i=1}^k \pi_j x_i = 0 \land \forall y \in T_1 \cup \ldots \cup T_{j-1}, \left| x \cap y \right| \leq k-2 \right\} \\
& T = T_1 \cup T_2 \cup \ldots \cup T_{k+1}
\end{align*}

We first show that with high probability for every $x \in K_k^n$ there exists an injection $\pi_i$ s.t.\ $x$ contains no $\left(k-2\right)$-set whose sum is $0$ under $\pi_i$.

\begin{prop}\label{pr:no_null_sums}
	The probability (over the choice of $\pi_1 , \ldots , \pi_{k+1}$) that for all $x \in K_k^n$ there exists some $1 \leq i \leq k+1$ s.t.\ none of the $\left(k-2\right)$-sets contained in $x$ sum to $0$ under $\pi_i$ is $1 - O \left( \frac{1}{n} \right)$.
\end{prop}

\begin{proof}
	Observe that for every $i$, the probability that $x \in K_k^n$ contains a $\left(k-2\right)$-set that sums to $0$ under $\pi_i$ is $O \left( \frac{1}{n} \right)$. Since the $\pi_i$s are independent the probability that this will hold for all the injections is $O \left( \frac{1}{n^{k+1}} \right)$. There are $O \left( n^k \right)$ edges, so by a union bound the probability that this will hold for at least one of them is at most $O \left( n^k \frac{1}{n^{k+1}} \right) = O \left( \frac{1}{n} \right)$.
\end{proof}

\begin{rmk}
	Although the condition in Proposition \ref{pr:no_null_sums} is satisfied only with probability $1 - O \left( \frac{1}{n} \right)$ rather than w.v.h.p.\ we'll still be able to find an $\steinParams$-design in $H$ w.v.h.p. This is because the probability space in Proposition \ref{pr:no_null_sums} refers only to the choice of $\pi_1 , \ldots , \pi_{k+1}$, and not the choice of $H$. So we may repeat the choice of the injections until the condition in Proposition \ref{pr:no_null_sums} is satisfied, and then proceed with the proof, where the remaining assertions hold w.v.h.p.\ even with this conditioning.
\end{rmk}

If we fix one of the injections $\pi_i$, the definitions in section \ref{ssec:template} make sense if we identify $x \in \left[ n \right]$ with $\pi_i x$. In particular, Propositions \ref{pr:vo_properties} and \ref{pr:absorbers} hold for every $x \in K_k^n$ containing no set of $k-2$ vertices whose sum (under $\pi_i$) is $0$. It's important to note that for each $i$ and each facet $f \in K_{k-1}^n$ the probability that $f$ is covered by $\pi_i$ is $O \left( p \right)$. Thus relatively few edges $x \in H$ for which $\sum_{\ell = 1}^{n} \pi_i x_\ell = 0$ aren't facet disjoint from $T_1 \cup \ldots \cup T_{i-1}$. This allows the probabilistic argument used to prove Proposition \ref{pr:absorbers} to work with minor modifications. As a consequence we obtain the following analogue of Proposition \ref{pr:absorbers}:

\begin{prop}
	For every $x \in K_k^n$ there exists some $i_x \in \left[ k+1 \right]$ s.t.\ there are $\Omega \left( n^{k - M \varepsilon} \right)$ absorbers for $x$ under $\pi_{i_x}$.
\end{prop}

For every $x \in K_k^n$ we let $\mathcal{A}_x$ be the set of absorbers for $x$ under $i_x$.

Proceeding to the Nibble, only minor modifications of the analysis in Appendix \ref{app:nibble} are needed to show that Proposition \ref{pr:leave} holds with item \ref{itm:affine_bounded_leave} (affine-boundedness of $L$) replaced with the condition, ``$\left( L , \pi_i \right)$ is $C$-affine bounded for every $i \in \left[ k+1 \right]$''.

As before, we apply Keevash's theorem to obtain $S$. We then apply Algorithm \ref{alg:RGA} to choose an absorber for every $x \in S$. The proof of Proposition \ref{pr:RGA} (that w.v.h.p.\ Algorithm \ref{alg:RGA} doesn't abort) goes through almost as before, where slight modifications are needed to take into account the fact that the absorbers for each $x \in S$ are w.r.t.\ $\pi_{i_x}$.

\section{Generalizations: Latin Squares, One Factorizations, and High-Dimensional Permutations}\label{sec:generalization}

Steiner systems are perhaps the simplest examples of $K_k$-decompositions of hypergraph families (specifically, of $K_{k-1}^n$). There are other hypergraph families one may wish to decompose, and for which the threshold problem is interesting. We mention a few here.

\begin{itemize}
	\item An order-$n$ \termdef{Latin square} is an $n \times n$ matrix in which each line and column is a permutation of $\left[ n \right]$. This is equivalent to a triangle (i.e.\ $K_3$)-decomposition of the complete tripartite graph with $n$ vertices in each part.
	
	\item An order-$n$ \termdef{one factorization} is an edge-coloring of the complete graph $K_n$ using $n-1$ colors (this exists iff $n$ is even). Equivalently, a one factorization is a triangle-decomposition of the graph join of $K_n$ with $n-1$ isolated vertices.
	
	\item An order-$n$ \termdef{$k$-dimensional permutation} has been defined by Linial and Luria \cite{linial2014upper} as an $\left[ n \right]^{k+1}$ $\left\{ 0,1 \right\}$-array in which every line (note that there are $k+1$ possible directions) contains a single one. For $k=1$ this is simply a permutation matrix and for $k=2$ this is equivalent to a Latin square. An order-$n$ $k$-dimensional permutation is equivalent to a $K_k$-decomposition of the complete $k$-partite $\left(k-1\right)$-uniform hypergraph with $n$ vertices in each part.
\end{itemize}

Just as with Steiner systems, for each of these families a lower bound on the threshold for their appearance is the threshold for the disappearance of uncovered facets, i.e.\ $\Theta \left( \frac{\log n }{n} \right)$. With the exception of $1$-dimensional permutations (which are just perfect matchings in bipartite graphs) to the best of our knowledge there are no published upper bounds. It seems likely that an analogue of Keevash's Theorem \ref{thm:Keevash} for these objects would imply an upper bound on the threshold of the form $n^{-\varepsilon}$. In particular, the argument from Section \ref{sec:proof_k4} should go through with minor modifications. The crux of the matter is the observation that our absorbers are $k$-partite, and therefore embed naturally in the hypergraphs above.

It is our hope that analogues of Keevash's result for the above objects will soon become available. In particular, adapting the proof of \cite{Ke15} (which is specialized to $\left(n,3,2\right)$-designs, a.k.a.\ Steiner triple systems) to Latin squares and one factorizations should be straightforward. Beyond the threshold problem, Keevash's result has advanced the understanding of random designs considerably (for example, see \cite{linial2015discrepancy} and \cite{kwan2016almost}), and we hope appropriate generalizations will do the same for other objects as well.

\section{Conjectures and Related Problems}\label{sec:future}

The most natural open problem related to this paper is the determination of the exact threshold for the appearance of Steiner systems in random hypergraphs. We've already mentioned the lower bound of $\Omega \left( \frac{\log n }{n} \right)$, obtained by considering uncovered facets. In the case of perfect matchings in hypergraphs, the lower bound arising from isolated vertices \cite{schmidt1983threshold} turned out to be correct \cite{johansson2008factors}. A natural, if bold, conjecture is therefore:

\begin{conj}\label{conj:steiner_thresh}
	The threshold for the appearance of an $\steinParams$-Steiner system in $\hnp$ is $\Theta \left( \frac{\log n}{n} \right)$.
\end{conj}

How far is our result from Conjecture \ref{conj:steiner_thresh}? The statement of Theorem \ref{thm:main} asserts nothing regarding the value of $\varepsilon$. Scanning the proof, we see that the main constraint on $\varepsilon$, at the end of the proof of Proposition \ref{pr:RGA}, is $\varepsilon < \frac{\delta}{2 M}$. Recall that $M = 2^{2k - 1} - 2^k + 1$ is the number of edges in an absorber and $n^{-\delta}$ is the upper bound on the density of $L$ (the facets uncovered by the Template and Nibble). In turn, the density of $L$ is bounded below by $n^{-a}$, where $a \in \left( 0,1 \right)$ is a non-explicit constant arising from Keevash's Theorem \ref{thm:Keevash}. Therefore, even in the best case, where $a=1$ (in reality, $a$ is far less), we would only be able to prove an upper bound of $n^{- \frac{1}{2 M} + \oone }$ on the threshold.

For the very brave, Conjecture \ref{conj:steiner_thresh} may be further strengthened by considering the random hypergraph process in $K_k^n$. In this model, we look at an evolving hypergraph in $\hkn$ in which the edges are added one by one in a uniformly random order. The hitting time for containing a Steiner system (i.e.\ the number of edges in the first graph in the process containing a Steiner system) is certainly no less than the hitting time for the disappearance of uncovered facets. Perhaps surprisingly, for $k = 2$ (perfect matchings in graphs) in almost all graph processes they are one and the same \cite{bollobas1985random}. On the other hand, for perfect matchings in hypergraphs, for which the threshold is known, it is not known whether the hitting time for isolated vertices is typically the same as that for the appearance of perfect matchings.

\begin{conj}
	In the random hypergraph process in $K_k^n$, asymptotically almost surely (as $n \rightarrow \infty$, $k$ fixed) a Steiner system appears at the very moment that all facets are covered.
\end{conj}

Another interesting direction is the study of fractional designs. A \termdef{fractional $K_k$-decomposition} of $H \in \hkn$ is a weight function $w : H \rightarrow \left[ 0,1 \right]$ s.t.\ for each $f \in K_{k-1} \left( H \right)$, $\sum_{f \subseteq e \in H} w \left( e \right) = 1$. The threshold for the appearance of fractional decompositions is bounded below by the disappearance of uncovered facets and above by the appearance of Steiner systems, as is the hitting time. More than a decade before the threshold for perfect matchings in random hypergraphs was determined, Krivelevich \cite{krivelevich1996perfect} showed that the hitting time for fractional matchings is almost always the same as the hitting time for the disappearance of isolated vertices. We conjecture:

\begin{conj}
	The threshold for the appearance of fractional $K_k$-decompositions in $\hnp$ is $\Theta \left( \frac{\log n}{n} \right)$.
\end{conj}

\begin{conj}\label{conj:fractional_hitting_time}
	In the random hypergraph process in $K_k^n$, asymptotically almost surely (as $n \rightarrow \infty$, $k$ fixed) the hitting time for the disappearance of uncovered facets and the appearance of fractional $K_k$-decompositions are the same.
\end{conj}

Determining whether a hypergraph contains a Steiner system is NP-hard \cite{colbourn1983embedding}. On the other hand, determining whether a hypergraph contains a fractional $K_k$-decomposition is a linear program, and hence solvable in polynomial time. We have performed numerical experiments indicating that, at least for $k = 3$, Conjecture \ref{conj:fractional_hitting_time} holds.

\subsection*{Acknowledgments}

The author wishes to thank his adviser, Nati Linial. His enthusiasm is contagious, and is matched by his patience and encouragement. Nati also helped write the introduction. It was Zur Luria's suggestion that Keevash's method should yield a threshold result. For this, and for hours of fun conversation, the author owes him a special thanks. Yuval Peled read the manuscript and made helpful comments. Last but not least, talking the proof through with Shanee Rosen benefited the exposition a great deal.

\appendix

\section{An Analysis of the R\"{o}dl Nibble}\label{app:nibble}

\newcommand{\bitesize}{{\left( 1 - \nu e^{-\nu k} \right)^k}}

In this section we prove Proposition \ref{pr:leave} by an application of the R\"odl Nibble. We assume $H \in \hkn$, $\pi : \left[ n \right] \hookrightarrow \F^*$ and $T \subseteq H$ have been chosen as described in Section \ref{ssec:template}. Our analysis is quite crude, and we make no attempt to recover the optimal constants.

The Nibble constructs a sequence of hypergraphs $H \supseteq G_0 \supseteq G_1 \supseteq \ldots \supseteq G_\tau$ as well as a sequence ${N_0 \subseteq N_1 \subseteq \ldots \subseteq N_\tau \subseteq G_0}$ of sets of facet-disjoint edges where for each $0 \leq t \leq \tau$, $T$, $G_t$, and $N_t$ are facet disjoint and $K_{k-1}^n = K_{k-1} \left( T \cup N_t \cup G_t \right)$. We call $L_t \coloneqq K_{k-1} \left( G_t \right)$ the set of uncovered facets (at stage $t$). The construction is as follows: Fix some sufficiently small $\nu = \nu \left( k \right) > 0$. Let $G_0 \subseteq H$ be the subhypergraph consisting of all edges that are facet disjoint from $T$. Let $N_0 = \emptyset$. Set $D_t = \frac{k \left| G_t \right|}{\left| L_t \right|}$. At each step take a random subset $S_t \subseteq G_t$, where each edge is retained with probability $\frac{\nu}{D_t}$, independently. Let $S_t' \subseteq S_t$ be the set of edges that are facet disjoint from all other edges in $S_t$. We set $N_{t+1} = N_t \cup S_t'$ and let $G_{t+1} \subseteq G_t$ be the set of edges that are facet disjoint from $N_{t+1}$.

\begin{prop}\label{pr:nibble_steps}
	Let $h \in \N$. There exists some $C = C \left( k \right) > 1$ and some $\mathcal{T} > 0$ s.t.\ w.v.h.p.\ (over the choice of $H$, $\pi$, and the internal randomness of the Nibble), for each $0 \leq t \leq \tau \coloneqq \lfloor \mathcal{T} \log n \rfloor$, the following hold, for $c_t = C^{t+1} p^{\frac{1}{3}}$:

	\begin{enumerate}
		\item $D_t = \left( 1 \pm c_t \right) \left( 1 - \nu e^{- k \nu} \right)^{\left( k - 1 \right) t} np$.
		
		\item $\left| L_t \right| = \left( 1 \pm c_t \right) \left( 1 - \nu e^{-k \nu } \right)^t \binom{n}{k - 1}$.
		
		\item $L_t$ is $\left( c_t , h \right)$-typical.
		
		\item $L_t$ is $C$-affine-bounded.
		
		\item $L_t$ is $k$-divisible.
	\end{enumerate}
\end{prop}

Proposition \ref{pr:leave} is an immediate consequence of Proposition \ref{pr:nibble_steps}, as one may take $N = N_t$ for some $t = \Omega \left( \log n \right)$. Intuitively, for every $t \leq \tau$, $L_{t+1}$ looks like a random subset of $L_t$ with density approximately $\left( 1 - \nu e^{-k \nu } \right)$. This ensures that the pseudo-randomness conditions hold, and that after $\log n$ steps the density of $L_t$ is polynomially small.

Proposition \ref{pr:nibble_steps} follows by inducing on $t$ and the next two lemmas, together with the observation that removing one $k$-divisible graph from another leaves a $k$-divisible graph. Therefore $L_t = K_{k-1}^n \setminus K_{k-1} \left( T \cup N_t \right)$ is $k$-divisible.

\begin{lem}\label{lem:induction_base}
	Let $h \in \N$. There exists some $C > 1$ s.t.\ w.v.h.p.\ (over the choice of $H$ and $\pi$) $G_0$ satisfies:
	\begin{enumerate}
		\item\label{itm:g_0_size} $\left| G_0 \right| = \left( 1 \pm C p \right) \binom{n}{k} p$.
		
		\item $\forall f \in K_{k-1} \left( G_0 \right)$, $\left| G_0 \left( f \right) \right| = \left( 1 \pm C p \right) D_0$.
		
		\item\label{itm:g_0_k-1_size} $\left| L_0 \right| = \left( 1 \pm C p \right) \binom{n}{k-1}$.
				
		\item $D_0 = \left( 1 \pm C p \right)np$.
		
		\item $L_0$ is $\left( C p , h \right)$-typical.
		
		\item $L_0$ is $C$-affine-bounded.
	\end{enumerate}
\end{lem}

\begin{lem}\label{lem:induction_step}
	Let $G \in \hkn$, $D = \frac{k \left| G \right|}{\left| K_{k-1} \left( G \right) \right|}$, and let
	$$
	1 > \alpha > \max \left\{ \frac{1}{\sqrt{D}} , \frac{1}{\left| K_{k-1} \left( G \right) \right|} , d \left( K_{k-1} \left( G \right) \right)^{-h} \sqrt{\frac{D}{n}} \right\} \log n , \beta > 0
	$$
	Assume $G$ satisfies:
	\begin{enumerate}
		\item\label{itm:regularity} $\forall f \in K_{k-1} \left( G \right)$, $\left| G \left( f \right) \right| = \left( 1 \pm \alpha \right) D$.
		
		\item $K_{k-1} \left( G \right)$ is $\left( \alpha , h\right)$-typical.
		
		\item $K_{k-1} \left( G \right)$ is $\beta$-affine-bounded.
		
		\item $D \geq 1$.
		
		\item $\left| G \right| \geq n^\frac{5}{2}$.
		
		\item $d \left( K_{k-1} \left( G \right) \right) \geq \frac{\log n}{n}$.
	\end{enumerate}
	Assume $\nu < \nu_0 \left( k \right)$ is sufficiently small and let $G'$ be the result of applying a single stage of the Nibble to $G$, where in the initial stage $S$ is a random subset of $G$ in which each edge is included with probability $\frac{\nu }{D}$, independently, and $S' \subseteq S$ is the set of edges facet-disjoint from all others. There exists some $C = C \left( k \right) > 1$ s.t.\ w.v.h.p. $G'$ satisfies:
	\begin{enumerate}[label={(\arabic*)}]
		\item\label{itm:conc_degree} $D' \coloneqq \frac{ k \left| G' \right| }{\left| K_{k-1} \left( G' \right) \right| } = \left( 1 \pm C \alpha \right) \left( 1 - \nu e^{ - k \nu } \right)^{k-1} D$.
		
		\item\label{itm:conclusion_reg} $\forall f \in K_{k-1} \left( G' \right)$, $\left| G' \left( f \right) \right| = \left( 1 \pm C \alpha \right) D'$.
		
		\item\label{itm:leave_cover} $K_{k-1} \left( G' \right) = K_{k-1} \left( G \right) \setminus K_{k-1} \left( S' \right)$.
		
		\item\label{itm:conclusion_size} $\left| G' \right| = \left( 1 \pm C \alpha \right) \bitesize \left| G \right|$
		
		\item\label{itm:conclusion_typical} $K_{k-1} \left( G' \right)$ is $\left( C \alpha, h \right)$-typical.
		
		\item\label{itm:conclusion_affine_bounded} $K_{k-1} \left( G' \right)$ is $\left( 1 + C \alpha \right) \beta$-affine-bounded.
	\end{enumerate}
\end{lem}

\begin{proof}[Proof of Lemma \ref{lem:induction_base}]
	The proof of all the items will proceed by conditioning on some property of $H$ that holds w.v.h.p., then calculating the expected value of some random variable over the choice of $\pi$, and applying Lemma \ref{lem:random_injection} to conclude that w.v.h.p.\ the random variable is close to its expectation.
	
	We'll show that the items hold with $C p$ replaced by $O \left( p \right)$. Once all the items have been proved we may take $C$ to be the maximum of the implicit constants.
	
	\begin{enumerate}
		\item By Chernoff bounds w.v.h.p.\ $\left| H \right| = \left( 1 \pm p \right) \binom{n}{k} p$, and furthermore w.v.h.p.\ every edge in $H$ intersects $O \left( np\right)$ other edges on a facet. Condition on these events occurring.
		
		The probability (over the choice of $\pi$) that $e \in H$ is also in $T$ is $\Theta \left( n \right)$. Thus $\E \left| T \right| = \Theta \left( \frac{\left| H \right|}{n} \right) = \Theta \left( n^{k-1} p \right)$. Since every element of $\F^*$ is contained in $O \left( n^{k-2} \right)$ edges of $T$, composing a transposition of $\F^*$ with $\pi$ affects $\left| T \right|$ by $O \left( n^{k-2} \right)$. Applying Lemma \ref{lem:random_injection}, w.v.h.p.\ $\left| T \right| = O \left( n^{k-1} p \right)$. $G_0$ consists of all edges of $H$ that are facet disjoint from $T$. Therefore:
		$$
		\left| G_0 \right| = \left| H \right| - \left| T \right| O \left( np \right) = \left( 1 \pm O \left( p \right) \right) \binom{n}{k} p
		$$
		
		\item
		\newcommand{\hf}{{H \left( f \right)}}
		
		Let $f = x_1 x_2 \ldots x_{k-1} \in K_{k-1}^n$. Set:
		\begin{gather*}
		\begin{align*}
		&S = \left\{ e \in H : f \subseteq e \right\}\\
		&S' = \left\{ e \in H : \left| e \cap f \right| = k-2 \land \left| e \cap S \right| = 2 \right\}
		\end{align*}
		\end{gather*}
		By Chernoff bounds, w.v.h.p.\ $\left| S \right| = \left| \hf \right| = \left( 1 \pm p \right) np$ and $\left| S' \right| = O \left( n^2 p^3 \right)$. Condition on these events and on $f \notin K_{k-1} \left( T \right)$.
		
		We want to show that w.v.h.p.\ there are $O \left(n p^2 \right)$ elements of $S$ that aren't facet-disjoint from $T$. Note that the event $f \notin T$ is determined by the values $\pi$ takes on $\hf$ (this is an abuse of notation. Formally, $\hf$ is the collection of singletons $\left\{ x \right\}$ s.t.\ $f \cup \left\{ x \right\} \in H$. Here we mean the values $\pi$ takes on $x$ s.t.\ $\left\{ x \right\} \in \hf$). We view $\pi$ as being chosen in three steps. First choose the values $\pi$ takes on $f$, then choose the values $\pi$ takes on $\hf$, while conditioning on $f \notin T$ (by choosing a random injection from $\hf$ to $\F^* \setminus \left( \pi \left( f \right) \cup \left\{ \pi x_1 + \pi x_2 + \ldots \pi x_{k-1} \right\} \right)$). Finally, choose the values $\pi$ takes on $\left[ n \right] \setminus \left( f \cup \hf \right)$ (by choosing a random injection from $\left[ n \right] \setminus \left( f \cup \hf \right)$ to $\F^* \setminus \pi \left( f \cup \hf \right)$).
		
		Observe that $T \cap S'$ is determined by the second step. Condition on the values $\pi$ takes on $f$. For every $e \in S'$ the probability that $e \in T$ is $\Theta \left( \frac{1}{n} \right)$. Therefore $\E \left| T \cap S' \right| = O \left( n p^3 \right)$. Furthermore, composing a transposition of $\F^* \setminus \left( \pi \left( f \right) \cup \left\{ \pi x_1 + \pi x_2 + \ldots \pi x_{k-1} \right\} \right)$ with $\pi|_{f \cup \hf}$ changes $\left| T \cap S' \right|$ by $O \left( 1 \right)$. Therefore, by Lemma \ref{lem:random_injection}, w.v.h.p.\ $\left| T \cap S' \right| = O \left( n p^3 \right)$. Each element of $T \cap S'$ intersects $O \left( 1 \right)$ elements of $S$ on a facet, so $O \left( n p^3 \right)$ elements of $S$ are removed from $G_0$ in this way. Call the remaining edges of $S$ the surviving edges.
		
		We now estimate how many of the surviving edges are removed from $G_0$ in the third stage of choosing $\pi$. A surviving edge $e \in S$ is removed from $G_0$ iff it intersects an edge from $T$ on a facet. Since $\pi \left( \hf \right)$ has already been determined and $e$ survived, this happens iff there is some $e' \in H \setminus S'$ that intersects $e$ on a facet and $e' \in T$. Since $e' \notin S'$, it has a vertex in $\left[ n \right] \setminus \left( f \cup \hf \right)$. Therefore the probability that $e' \in T$ is $O \left( \frac{1}{n} \right)$. Since $e$ intersects $O \left(np\right)$ edges on a facet, the probability that $e \notin K_{k-1} \left( G_0 \right)$ is $O \left( p \right)$. Therefore the expected number of edges removed from $H \left( f \right)$ in this way is $O \left( \left| H \left( f \right)\right| p \right) = O \left( np^2 \right)$. Composing a transposition of $\left[n\right] \setminus \pi \left( f \cup S \right)$ with $\pi$ alters this estimate by $O \left( 1 \right)$. Applying Lemma \ref{lem:random_injection} we conclude that w.v.h.p.\ the number of edges removed in this way is $O \left( np^2 \right)$. Therefore:
		$$
		\left| G_0 \left( f \right) \right|
		= \left| H \left( f \right) \right| - O \left( np^2 \right)
		= \left( 1 \pm O \left( p \right) \right) np
		$$
		
		\item By the proof of the previous item, each $f \in K_{k-1}^n \setminus K_{k-1} \left( T \right)$ is contained in at least one (in fact, many) edges of $G_0$. Therefore, $L_0 = K_{k-1} \left( G_0 \right) = K_{k-1}^n \setminus K_{k-1} \left( T \right)$. In the proof of item \ref{itm:g_0_size} we showed that $\left| T \right| = O \left( n^{k-1} p \right)$. Thus:
		$$
		\left| L_0 \right| = \left| K_{k-1}^n \right| - k \left| T \right| = \binom{n}{k-1} - O \left( n^{k-1} p \right) = \left( 1 \pm O \left( p \right) \right) \binom{n}{k-1}
		$$
		
		\item Items \ref{itm:g_0_size} and \ref{itm:g_0_k-1_size} occur w.v.h.p., in which case:
		$$
		D_0 = \frac{k \left| G_0 \right|}{\left| K_{k-1} \left( G_0 \right) \right|}
		= \frac{k \left( 1 \pm O \left( p \right) \right) \binom{n}{k} p}{\left( 1 \pm O \left( p \right)\right) \binom{n}{k-1}}
		= \left( 1 \pm O \left( p \right) \right) np
		$$
		
		\item First observe that item \ref{itm:g_0_k-1_size} implies $d \left( K_{k-1} \left( G_0 \right) \right) = \left( 1 - O \left( p \right) \right)$.
		
		As shown previously, w.v.h.p.\ $K_{k-1} \left( H \right) = K_{k-1}^n$, in which case $K_{k-1} \left( H \right)$ is $\left( O \left( \frac{1}{n} \right) , h \right)$-typical. Furthermore, by a Chernoff bound, w.v.h.p.\ every $g \in K_{k-2}^n$ is contained in $O \left( n^2 p\right)$ edges of $H$. Condition on these occurrences.
		
		Let $g \in K_{k-2}^n$. We first bound the number of edges in $T$ that contain $g$. Denote this quantity by $X$. Conditioning on $\pi \left( g \right)$, the probability that a given edge in $H$ will be in $T$ is $O \left(\frac{1}{n}\right)$. Therefore $\E X = O \left( np \right)$. Composing a transposition of $\F^* \setminus \pi \left( g \right)$ with $\pi$ affects this estimate by $O \left( 1 \right)$. Therefore, by Lemma \ref{lem:random_injection}, w.v.h.p.\ $X = O \left( np \right)$. Therefore, w.v.h.p., every element of $K_{k-2}^n$ is contained in $O \left( np \right)$ facets in $K_{k-1} \left( T \right)$.
		
		Now, let $\ell \leq h$ and let $g_1,g_2 , \ldots ,g_\ell \in K_{k-2}^n$ be distinct. We have:
		$$
		\left| \bigcup_{i \in \left[ \ell \right]} K_{k-1} \left( T \right) \left( g_i \right) \right| = \ell O \left( np \right) = O \left( np \right)
		$$
		Therefore:
		\begin{gather*}
		\begin{align*}
			\left| \bigcap_{i \in \left[ \ell \right]} K_{k-1} \left( G_0 \right) \left( g_i \right) \right|
			&\geq \left| \bigcap_{i \in \left[ \ell \right]} K_{k-1} \left( H \right) \left( g_i \right) \right| - \left| \bigcup_{i \in \left[ \ell \right]} K_{k-1} \left( T \right) \left( g_i \right) \right|\\
			&= \left( 1 - O \left(\frac{1}{n}\right)\right) n - O \left( np \right)
			= \left( 1 \pm O \left( p \right) \right) n\\
			&= \left( 1 \pm O \left( p \right) \right) d \left( K_{k-1} \left( G_0 \right) \right)^\ell n
		\end{align*}
		\end{gather*}
		as desired.
		
		\item W.v.h.p.\ $K_{k-1} \left( H \right) = K_{k-1}^n$ implying that $K_{k-1} \left( H \right)$ is $\alpha$-affine-bounded for some $\alpha = O \left( 1 \right)$. Since $L_0 \subseteq K_{k-1} \left( H \right)$ and by item \ref{itm:g_0_k-1_size} w.v.h.p.\ $\left| L_0 \right| = \Omega \left( \left| K_{k-1} \left( H \right) \right| \right)$, $L_0$ is $\beta$-affine-bounded for some $\beta = O \left( 1 \right)$.
	\end{enumerate}	

\end{proof}

\begin{proof}[Proof of Lemma \ref{lem:induction_step}]
	\newcommand{\oalpha}{{O \left( \alpha \right)}}
	Recall that $S \subseteq G$ is a random subset where each edge is included with probability $\frac{\nu}{D}$ and $S' \subseteq S$ is the set of all edges in $S$ that are facet disjoint from all other edges in $S$.
	
	Let $e \in G$. We first estimate $\Prob \left[ e \in S' \right]$. Since by item \ref{itm:regularity} $e$ intersects $\left( 1 \pm \alpha \right)kD -\left( k - 1 \right) = \left( 1 \pm \oalpha \right) D$ other edges on a facet, we have:
	\begin{equation}\label{eq:nibble_prob}
		\Prob \left[ e \in S' \right]
		= \frac{\nu}{D} \left(1 - \frac{\nu}{D}\right)^{\left(1 \pm \oalpha \right) kD}
		= \left( 1 \pm \oalpha \right) \frac{\nu e^{- \nu k }}{D}
	\end{equation}
	This allows us to estimate the probability that $e \in G'$. We first write:
	$$
	\Prob \left[ e \in G' \right] = \Prob \left[ e \in G' \given e \notin S' \right] \cdot \Prob \left[ e \notin S' \right] + \Prob \left[ e \in G' \given e \in S' \right] \cdot \Prob \left[ e \in S' \right]
	$$
	Observe that by definition, if $e \in S'$ then $e \notin G'$, so the second term contributes $0$ to the sum. We next estimate the first term. Write:
	\begin{gather*}
	\begin{align*}
	\Prob \left[ e \in G' \given e \notin S' \right]
	= &\Prob \left[ e \in G' \given e \notin S' \land e \in S \right] \cdot \Prob \left[ e \in S \given e \notin S' \right]\\
	&+ \Prob \left[ e \in G' \given e \notin S' \land e \notin S \right] \cdot \Prob \left[ e \notin S \given e \notin S' \right]
	\end{align*}
	\end{gather*}
	$\Prob \left[ e \in S \given e \notin S' \right] \leq \frac{\Prob \left[ e \in S \right]}{\Prob \left[ e \notin S' \right]} \leq \frac{\Prob \left[ e \in S \right]}{\Prob \left[ e \notin S \right]} \leq \frac{1}{D} = \oalpha$. Hence:
	\begin{gather*}
	\begin{align*}
	\Prob & \left[ e \in G' \given e \notin S' \right] \cdot \Prob \left[ e \notin S' \right]\\
	& = \Prob \left[ e \in G' \given e \notin S' \land e \notin S \right] \cdot \Prob \left[ e \notin S \given e \notin S' \right] \left( 1 \pm \oalpha \right) \pm \oalpha\\
	&= \Prob \left[ e \in G' \given e \notin S \right] \cdot \left( 1 \pm \oalpha \right) \pm \oalpha
	\end{align*}
	\end{gather*}
	Finally, conditioning on $e \notin S$, $e \in G'$ iff none of the $k$ facets contained in $e$ are covered by an edge in $S'$. For each of the $k$ facets, the event that it is covered depends only on the $\left( 1 \pm \oalpha \right) D$ edges containing it, with the exception of $e$ itself. Since no edge besides $k$ contains two of these facets, these $k$ events are independent. Furthermore, each facet is covered by $S'$ iff exactly one of the edges containing it is in $S$. Therefore:
	$$
	\Prob \left[ e \in G' \given e \notin S \right]
	= \left( 1 \pm \oalpha \right) \bitesize
	$$
	Hence:
	$$
	\Prob \left[ e \in G' \right] = \left( 1 \pm \oalpha \right) \bitesize
	$$
	
	Therefore:
	$$
	\E \left[ \left| G' \right| \right] = \left( 1 \pm \oalpha \right) \bitesize \left| G \right|
	$$
	$\left|G'\right|$ is a Boolean function of the i.i.d.\ Bernoulli random variables indicating whether the edges of $G$ are in $S$. Changing one of these variables affects $\left|G'\right|$ by $O \left( D \right)$. Applying Lemma \ref{lem:boolean_concentration}, w.v.h.p.\ $ \left| G' \right| = \left( 1 \pm \oalpha \right) \bitesize \left| G \right|$. This proves item \ref{itm:conclusion_size}.
	
	We now show that w.v.h.p.\ for all $f \in K_{k-1} \left( G \right) \setminus K_{k-1} \left( S' \right)$:
	$$
	\left| G' \left( f \right) \right| = \left( 1 \pm \oalpha \right) \left( 1 - \nu e^{-k \nu} \right)^{k-1} D
	$$
	This will imply items \ref{itm:conc_degree}, \ref{itm:conclusion_reg}, and \ref{itm:leave_cover}. Let $f \in K_{k-1} \left( G \right)$. Then $\left| G \left( f \right) \right| = \left( 1 \pm \alpha \right)D$. Consider an edge $f \subseteq e \in G$. $e$ intersects $\left( 1 \pm \oalpha \right)kD$ other edges on a facet, and $ \left( 1 \pm O \left( \alpha \right) \right) D$ of these contain $f$. Therefore, conditioning on $f \notin K_{k-1} \left( S' \right)$, by a calculation similar to the last, we have:
	$$
	\Prob \left[ e \in G' \given f \notin K_{k-1} \left( S' \right) \right]
	= \left( 1 \pm \oalpha \right) \left( 1 - \nu e^{- k \nu} \right)^{k-1}
	$$
	Therefore:
	\begin{gather*}
		\begin{align*}
			\E \left[ \left| G' \left( f \right) \right| \given f \notin K_{k-1} \left( S' \right) \right]
			&= \left( 1 \pm \oalpha \right) \left( 1 - \nu e^{- k \nu} \right)^{k-1} \left| G \left( f \right) \right|\\
			&= \left( 1 \pm \oalpha \right) \left( 1 - \nu e^{- k \nu} \right)^{k-1} D
		\end{align*}
	\end{gather*}
	We'd like to apply Lemma \ref{lem:boolean_concentration} to obtain measure concentration of $\left| G' \left( f \right) \right|$. $\left| G' \left( f \right) \right|$ is a Boolean function of the $\left| G \right|$ variables indicating whether the edges of $G$ are in $S$. However, these variables are too numerous for Lemma \ref{lem:boolean_concentration} to be useful. To overcome this, we observe that $\left| G' \left( f \right) \right|$ is actually determined by a small subset of these variables. Set:
	$$
	X = \bigcup_{f \subseteq e \in G } \bigcup_{f' \in K_{k-1} \left(e\right)} \left\{ e' \in G : f' \subseteq e' \land f \nsubseteq e' \right\}
	$$
	
	If we condition on $f \notin K_{k-1} \left( S' \right)$, the value of $\left| G' \left( f \right) \right|$ depends only on the $O \left( D^2 \right)$ i.i.d.\ Bernoulli random variables that indicate whether the edges in $X$ are in $S$. Changing the value of one of these variables affects $\left| G' \left( f \right)\right|$ by $O \left( 1 \right)$. Applying Lemma \ref{lem:boolean_concentration} we obtain that w.v.h.p.\ $\left| G' \left( f \right)\right| = \left( 1 \pm \oalpha \right) \left( 1 - \nu e^{- k \nu} \right)^{k-1} D$.
	
	Before proceeding we observe that items \ref{itm:conc_degree} and \ref{itm:conclusion_size} imply:
	$$
	d \left( K_{k-1} \left( G' \right) \right) = \left( 1 \pm \oalpha \right) \left( 1 - \nu e^{-k \nu} \right) d \left( K_{k-1} \left( G \right) \right)
	$$
	
	We turn to item \ref{itm:conclusion_typical}. Let $s_1 , s_2 , \ldots , s_\ell \in K_{k-2}^n , \ell \leq h$ be distinct. Then, by the inductive hypothesis, $\left| \cap_{i \in \left[ \ell \right]} K_{k-1} \left( G \right) \left( s_i \right) \right| = \left( 1 \pm \alpha \right) d \left( K_{k-1} \left( G \right) \right)^\ell n$. Let $x \in \cap_{i \in \left[ \ell \right]} K_{k-1} \left( G \right) \left( s_i \right)$. We first estimate $\Prob \left[ x \in \cap_{i \in \left[ \ell \right]} K_{k-1} \left( G' \right) \left( s_i \right) \right]$. Set $f_i = s_i \cup \left\{ x \right\} \in K_{k-1} \left( G \right)$. $x \in \cap_{i \in \left[ \ell \right]} K_{k-1} \left( G' \right) \left( s_i \right)$ iff none of $f_1 , f_2 , \ldots , f_\ell$ are covered by $S'$. The event $f_i \in K_{k-1} \left( S' \right)$ depends only on the i.i.d.\ Bernoulli random variables indicating whether the edges containing $f_i$ are in $S$. There are at most $O \left( 1 \right)$ edges in $G$ covering more than one of $f_1 , f_2 , \ldots , f_\ell$ (because given any two distinct facets there is at most one edge containing both of them). Denote by $M$ the set of these edges, and observe that $\Prob \left[ M \cap S \neq \emptyset \right] \leq M \frac{\nu}{D} = \oalpha$. We write:
	\begin{gather*}
	\begin{align*}
	\Prob \left[ x \in \cap_{i \in \left[ \ell \right]} K_{k-1} \left( G' \right) \left( s_i \right) \right]
	= & \Prob \left[ x \in \cap_{i \in \left[ \ell \right]} K_{k-1} \left( G' \right) \left( s_i \right)  \given M \cap S = \emptyset \right] \cdot \Prob \left[ M \cap S = \emptyset \right]\\
	& + \Prob \left[ x \in \cap_{i \in \left[ \ell \right]} K_{k-1} \left( G' \right) \left( s_i \right)  \given M \cap S \neq \emptyset \right] \cdot \Prob \left[ M \cap S \neq \emptyset \right]\\
	= & \left( 1 \pm \oalpha \right) \Prob \left[ \left\{ f_1 , f_2 , \ldots , f_\ell \right\} \cap K_{k-1} \left( S' \right) = \emptyset \given M \cap S = \emptyset \right]\\
	& + \oalpha
	\end{align*}
	\end{gather*}
	Conditioning on $M \cap S = \emptyset$, the events $ \left\{ f_i \in K_{k-1} \left( S' \right) \right\}_{i \in \left[ \ell \right]}$ are independent, and each occurs with probability $\left( 1 \pm \oalpha \right) \nu e^{-k \nu}$. Therefore:
	\begin{equation*}
	\Prob \left[ x \in \cap_{i \in \left[ \ell \right]} K_{k-1} \left( G' \right) \left( s_i \right) \right]
	= \left( 1 \pm \oalpha \right) \left( 1 - \nu e^{- k \nu} \right)^\ell
	\end{equation*}
	Hence:
	\begin{gather*}
	\begin{align*}
	\E \left| \bigcap_{i \in \left[ \ell \right]} K_{k-1} \left( G' \right) \left( s_i \right) \right|
	& = \left( 1 \pm \oalpha \right) \left( 1 - \nu e^{-k \nu} \right)^\ell d \left( K_{k-1} \left( G \right) \right)^\ell n\\
	& = \left( 1 \pm \oalpha \right) d \left( K_{k-1} \left( G' \right) \right)^\ell n
	\end{align*}
	\end{gather*}
	
	$\left| \bigcap_{i \in \left[ \ell \right]} K_{k-1} \left( G' \right) \left( s_i \right) \right|$ is a function of the $O \left( n D^2 \right)$ i.i.d.\ Bernoulli random variables indicating whether the edges containing $s_i$ are in $S$. Changing the value of one of these variables affects the function by $O \left( 1 \right)$. Applying Lemma \ref{lem:boolean_concentration} we obtain, w.v.h.p.:
	$$
	\left| \bigcap_{i \in \left[ \ell \right]} K_{k-1} \left( G' \right) \left( s_i \right) \right|
	= \left( 1 \pm \oalpha \right) d \left( K_{k-1} \left( G' \right) \right)^\ell n
	$$
	Hence item \ref{itm:conclusion_typical} holds w.v.h.p.
	
	Finally, we need to show item \ref{itm:conclusion_affine_bounded}, i.e.\ that w.v.h.p.\ $K_{k-1} \left( G' \right)$ is affine-bounded. Since an affine space of dimension $d$ can be partitioned into $\left| \F \right|^{d-1}$ affine spaces of dimension $1$, it's enough to show that for every affine space $A \subseteq \F^{k-1}$ of dimension $1$:
	\begin{gather*}
	\begin{align*}
	\left| A \cap \pi \left( K_{k-1} \left( G' \right) \right) \right|
	&\leq \left( 1 + \oalpha \right) \left| A \cap \pi \left( K_{k-1} \left( G \right) \right) \right| \frac{\left| K_{k-1} \left( G' \right) \right|}{\left| K_{k-1} \left( G \right) \right|}\\
	&= \left( 1 + \oalpha \right) \left| A \cap \pi \left( K_{k-1} \left( G \right) \right) \right| \left( 1 - \nu e^{-k \nu} \right)
	\end{align*}
	\end{gather*}
	
	Let $f' \in A \cap \pi \left( K_{k-1} \left( G \right) \right)$. There is a unique $f \in K_{k-1} \left( G \right)$ s.t.\ $\pi \left( f \right) = f'$. We have:
	$$
	\Prob \left[ f' \in A \cap \pi \left( K_{k-1} \left( G' \right) \right) \right]
	= \Prob \left[ f \in K_{k-1} \left( G' \right) \right]
	= \left( 1 \pm O \left( \alpha \right)\right) \left( 1 - \nu e^{- k \nu} \right)
	$$
	Therefore:
	$$
	\E \left[ \left| A \cap \pi \left( K_{k-1} \left( G' \right) \right) \right| \right]
	= \left| A \cap \pi \left( K_{k-1} \left( G \right) \right) \right| \left( 1 \pm O \left( \alpha \right) \right) \left( 1 - \nu e^{-k \nu} \right)
	$$
	$\left| A \cap \pi \left( K_{k-1} \left( G' \right) \right) \right|$ is determined by the $O \left( \left| A \cap \pi \left( K_{k-1} \left( G \right) \right) \right| D^2 \right) = O \left( n D^2 \right)$ i.i.d.\ Bernoulli random variables that indicate whether edges containing facets from $A$, and those intersecting them on facets, are in $S$. Changing the value of one of the variables affects the count by $O \left( 1 \right)$. Applying Lemma \ref{lem:boolean_concentration}, w.v.h.p.:
	\begin{gather*}
		\begin{align*}
			\left| A \cap \pi \left( K_{k-1} \left( G' \right) \right) \right|
			&= \left| A \cap \pi \left( K_{k-1} \left( G \right) \right) \right| \left( 1 \pm \oalpha \right) \left( 1 - \nu e^{-k \nu} \right)\\
		\end{align*}
	\end{gather*}

\end{proof}

\bibliography{latin_box}
\bibliographystyle{amsplain}

\end{document}